\documentclass[11pt]{article}

  \usepackage{mathtools, amsthm, amssymb, eucal}
  \usepackage{microtype}
  \usepackage{verbatim}
  \usepackage{indentfirst} % Indents first line of first section paragraph
  \usepackage{graphicx}
  \usepackage[all]{xy}
  \usepackage[percent]{overpic}
  \usepackage{booktabs}
  \usepackage{alltt} % for verbatim program listings--still executes LaTeX commands
  \usepackage{spverbatim}
  \usepackage{geometry} % to change the page dimensions
\usepackage{hyperref} % SHOULD BE LAST PACKAGE IN THE LIST!

    \theoremstyle{plain}
    \newtheorem*{thm}{Theorem}
    \newtheorem*{prop}{Proposition}
    \newtheorem*{lem}{Lemma}
    \newtheorem*{cor}{Corollary}
    \newtheorem*{exer}{Exercise}
    \theoremstyle{remark}
    
  \newcommand{\abs}[1]{\left\lvert #1 \right\rvert}  % single bar abs value symbol

  \newcommand{\LS}{$ \mathcal{LS} $}
  \newcommand{\tinyzero}{\scalebox{0.47}{0}}
   \title{\bfseries Shortest Paths on Cubes }
   \author{Richard Goldstone, Rachel Roca, and Robert Suzzi Valli\\Department of Mathematics\\Manhattan College, Riverdale, NY}
   \date{March 7, 2020} % Activate to display a given date or no date (if empty),
\begin{document}

\maketitle

\begin{abstract}
	In 1903, noted puzzle-maker Henry Dudeney published \emph{The Spider and the Fly} puzzle, which asks for the shortest path along the surfaces of a square prism between two  points (source and target) located on the square faces, and surprisingly showed that the shortest path traverses five faces. Dudeney's source and target points had very symmetrical locations; in this article, we allow the source and target points to be anywhere in the interior of opposite faces, but now require the square prism to be a cube.  In this context, we find that, depending on source and target locations, a shortest path can traverse either three or four faces, and we investigate the conditions that lead to four-face solutions and estimate the probability of getting a four-face shortest path.  We utilize a combination of numerical calculations, elementary geometry, and transformations we call corner moves of cube unfolding diagrams
\end{abstract}

\section{Introduction}
On June 14, 1903, noted puzzle-maker Henry Dudeney published \emph{The Spider and the Fly} puzzle in the \emph{Weekly Dispatch,} a British newspaper, and later included the puzzle in \emph{The Canterbury Puzzles,} his 1908 anthology \cite{heD}.  The puzzle asked for the shortest distance along the surfaces of a room between the positions~$A$ (spider) and~$B$ (fly) indicated in Figure~\ref{f:spidfly}, taken from \emph{The Canterbury Puzzles.} (The points~$ A $ and~$ B $ are each one foot from an edge on the vertical center line of the $ 12\,\text{ft} \times 12\,\text{ft} $ square faces.)
\begin{figure}[h]
\centering{
\begin{overpic}[scale=0.13]{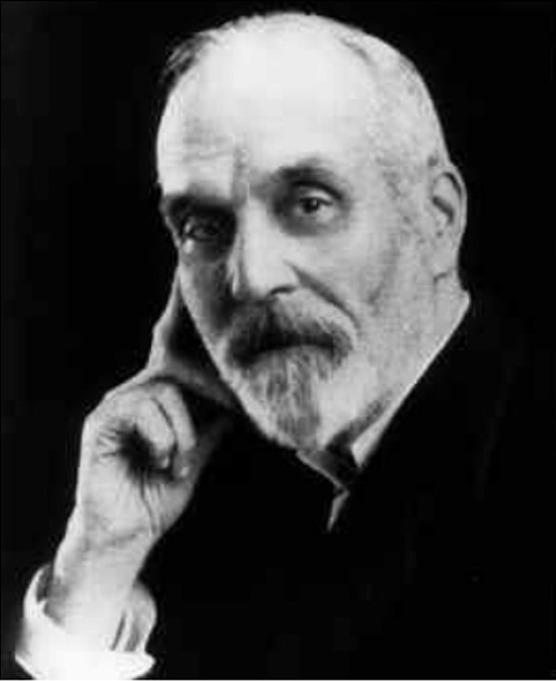}
	\put(6,-10){\footnotesize Henry Dudeney}
\end{overpic}	
\qquad
\begin{overpic}[scale=0.39]{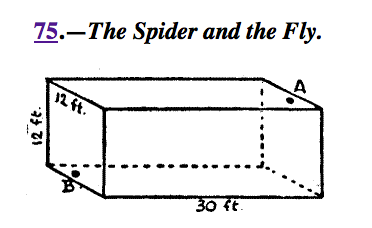}\end{overpic}
}
\caption{Dudeney's spider and fly problem \label{f:spidfly}}
\end{figure} 

Dudeney's solution (Figure~\ref{f:MVdata}\,(a)), as well as the analyses of analogous shortest path problems, involve cutting along some of the edges of the room and {\it unfolding} the resulting figure (which must be a ``single piece") into the plane. Unfoldings of this type are illustrated in  Figure~\ref{f:MVdata}\,(a) and analyzed in~\cite{rGrSV}.  The spider and fly points in each unfolding are joined with a straight line whenever possible, and the shortest of the straight lines is chosen. Dudeney's option number~4 in Figure~\ref{f:MVdata}\,(a) shows, surprisingly, that a route traversing five faces of the rectangular solid is actually the shortest path.

This problem has passed the test of time.  An internet search will reveal various video expositions, and it has been the subject of a 2012 Classroom Capsule in The College Mathematics Journal by K.~E.~Mellinger and R.~Viglione~\cite{kMrV}.  In that article, the authors retain the symmetry of Dudeney's spider-fly locations and find that a shortest path may traverse either three or four faces rather than five, depending on the dimensions of the solid.   Figure~\ref{f:MVdata}\,(b) summarizes, in terms of rectangular solid dimensions~$ x \times x \times y $ ($ x>2 $),  Mellinger and Viglione's results on the number of faces traversed by a shortest path linking Dudeney's spider-fly locations.  
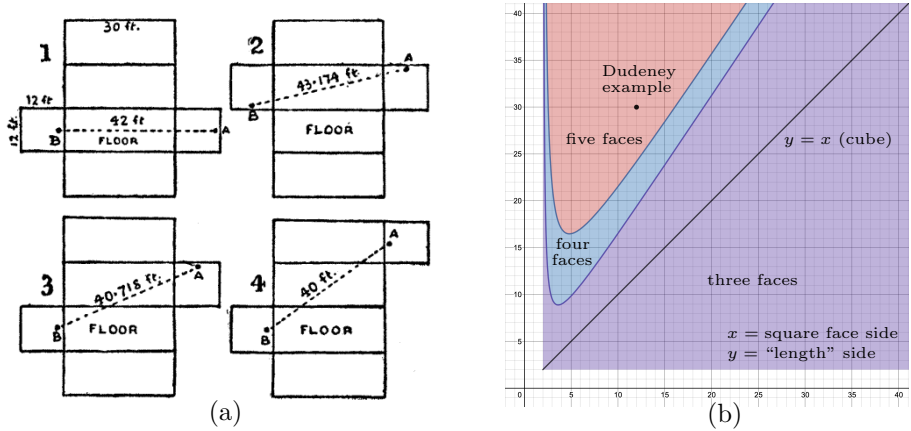
\begin{figure}[h]
	\centering{
		\begin{overpic}[scale=0.33]{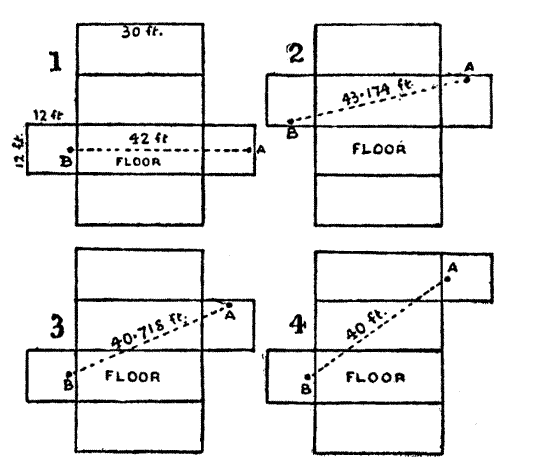}
			\put(45,-3){\small (a)}
		\end{overpic}
		\quad
		\begin{overpic}[scale=0.2]{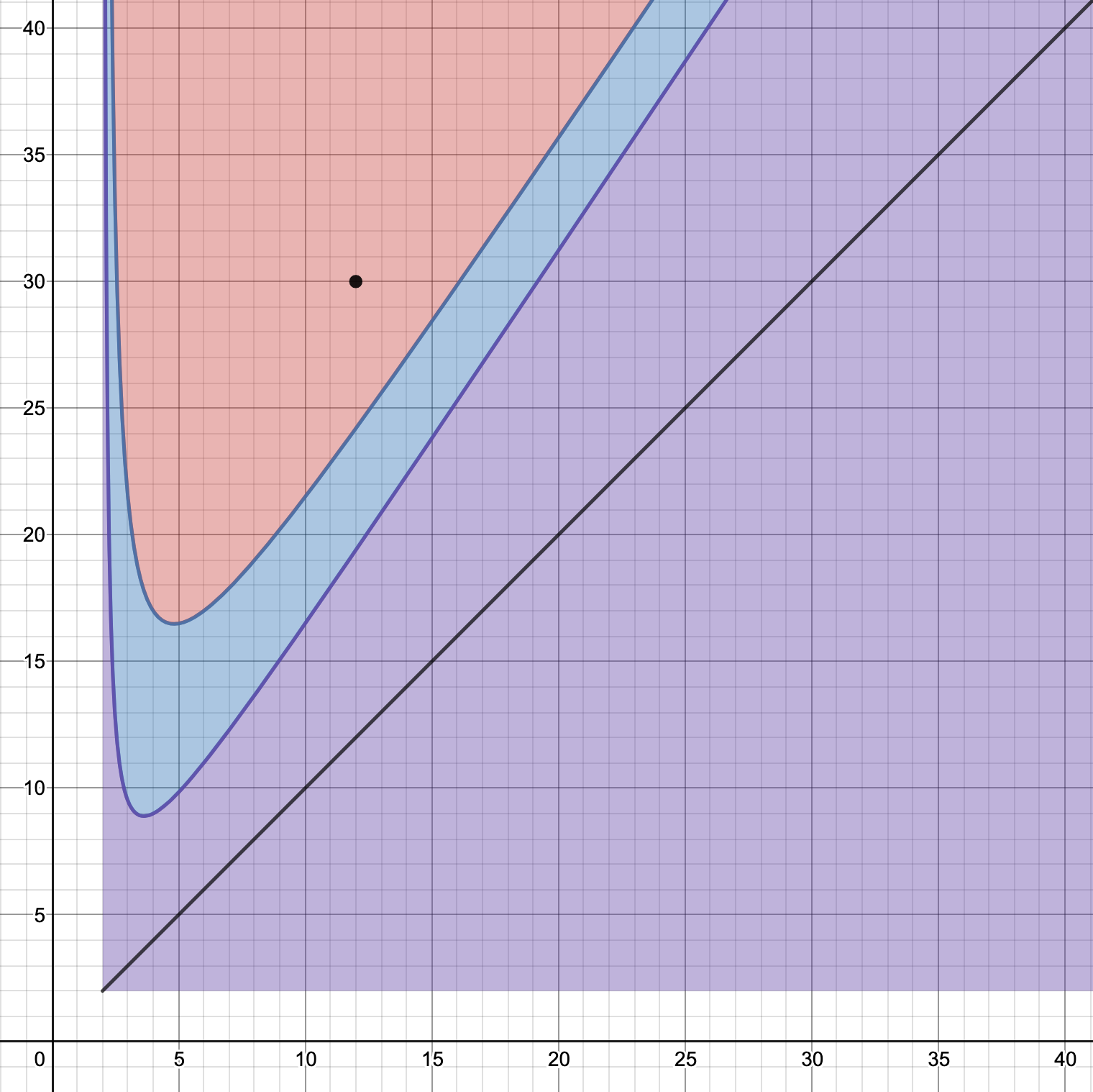} %\put(){}
			\put(24,82){\tiny Dudeney }
			\put(24,78){\tiny  example}
			\put(15,65){\tiny five faces}
			\put(12.5,39){\tiny four }
			\put(11.5,35){\tiny  faces}
			\put(50,30){\tiny three faces}
			\put(69,65){\tiny $ y=x $ (cube)}
			\put(55,17){\tiny $ x=\text{square face side} $}
			\put(55,12){\tiny $ y=\text{``length'' side} $}
			\put(50,-4){\small (b)}
		\end{overpic}	
	}
	\caption{Solutions and number of faces traversed \label{f:MVdata}}
\end{figure} 

 The symmetry of the spider-fly locations used by Dudeney and Mellinger-Viglione restrict the complexity of the shortest path problem. As a next step, we consider unconstrained spider-fly positions in the interiors of opposite faces, i.e.\ not on edges or at vertices.  In order to focus on the effects of these general positions,  we have eliminated the role of the solid's ``rectangularity'' by working only with cubes.  In Figure~\ref{f:MVdata}\,(b), the region containing the line~$ y=x $ shows that if we stick with the Dudeney  spider-fly positions in a cube,  the shortest path will always be a 3-face path regardless of the cube dimensions. So, for the cube case, any ``interesting'' phenomena will have to occur for non-Dudeney spider-fly locations. Recent published works have studied shortest path problems on regular polyhedra, including the cube~\cite{jAdA,dDvDcTjY,dF2,dF1, dFeF}. With different goals, they use more advanced combinatorial and geometric techniques than we use here.

From now on, we refer to the location of the spider as the \emph{source} and the location of the fly as the \emph{target} of the path.  Dudeney and Mellinger-Viglione have the four unfoldings of Figure~\ref{f:MVdata}\,(a) to consider, but in principle we have to deal with twelve unfoldings---see the right-hand diagram of Figure~\ref{f:unfgrid}.  To handle this additional complexity, we've developed conceptual tools and have relegated long but elementary algebraic computations to computer algebra systems.

\begin{figure}[h]
	\centering{
		\begin{overpic}[scale=0.15]{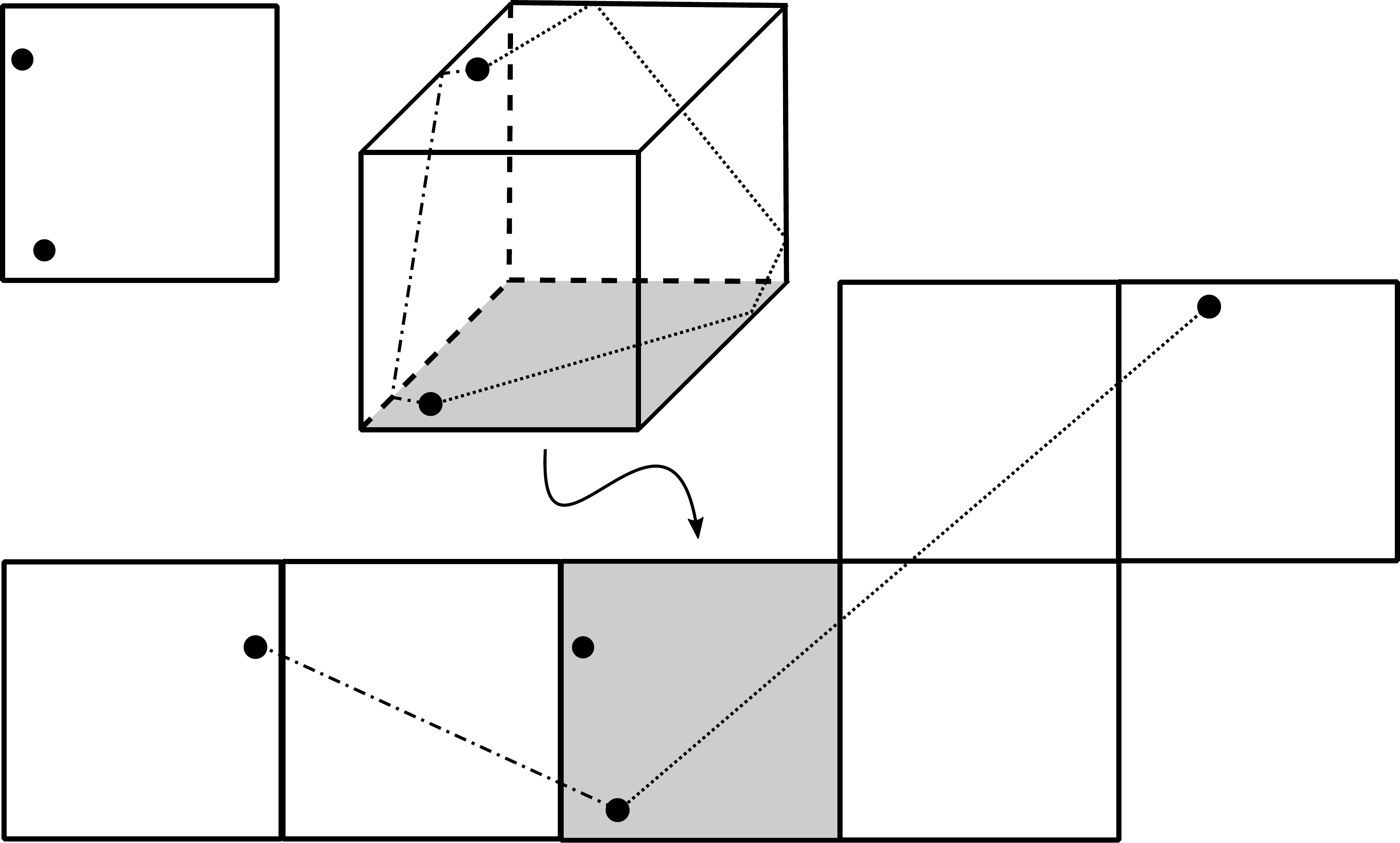} %\put(){\tiny }
			\put(1.5,35){\footnotesize Top view}
			\put(5,41.5){\tiny $S$ }
			\put(3,55.3){\tiny $T$ }
			\put(31,32.5){\tiny $S$}
			\put(32,51){\tiny $T$ }	
			\put(13,15.5){\tiny $T_1$ }
			\put(43,15.5){\tiny $T$}
			\put(43,4){\tiny $S$ }
			\put(88,36){\tiny $T_2$ }	
			\put(47,1){\tiny bottom}
			\put(71,1){\tiny right}
			\put(71,21){\tiny back}
			\put(93.5,21){\tiny top}
			\put(21,1){\tiny left}
			\put(1,1){\tiny top}
		\end{overpic}	
	}
	\caption{A grid of two partial unfoldings, each with a locally shortest path   \label{f:ur1}}
\end{figure}
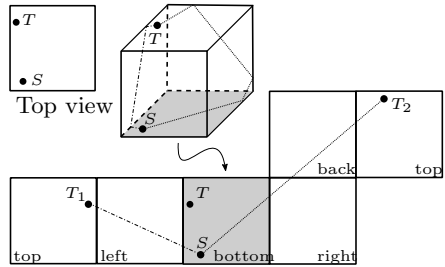  

Figure~\ref{f:ur1} illustrates some constructs.  Rather than following Dudeney and trying to separately depict the twelve unfoldings with the source and target points properly situated in each one, we use a standard construction, combining all the unfoldings into a single planar array of \emph{partial} unfoldings that we call an \emph{unfolding grid}.  (The term ``partial'' emphasizes that we only need the portion of an unfolding whose faces contain a straight line path from source to target.)  We can think of the unfolding grid as being generated by ``rolling'' the cube over its edges and keeping track of the face that lands in the plane. This model provides geometric constraints that automate proper location of the source and target points in the various unfoldings. 

We start with the cube in a fixed orientation with the source face down and resting on the central square of the unfolding grid. We refer to this central square as the \emph{base face} of the unfolding grid.  In Figure~\ref{f:ur1}, the source and target points are marked ``$ S $'' and ``$ T $''. The bottom (source) face of the cube is shaded, and the unfolding grid base face where the bottom face of the cube rests is  shaded to correspond.  Two paths are illustrated, on the cube and on the unfolding grid,  where the unfolded images of~$ T $ are labeled $ T_1 $ and~$ T_2 $. 

It can happen that the straight line from source to target point image is not contained in the interior of its unfolding, instead running through a vertex or having a portion outside the sequence of rolled faces in the unfolding diagram (see Figure~\ref{f:pseudp} for an example).   We refer to such paths as \emph{pseudopaths}. In the next sections, we shall see that adjustments we call \emph{corner moves} replace pseudopaths with paths in the interior of an unfolding, and these new paths are always shorter than the original pseudopath, so a pseudopath can never be a solution to the shortest path problem on the cube surface. We call  the paths that are not pseudopaths \emph{locally shortest paths}, or \emph{\LS-paths,} because each is the shortest path from source to target in the interior of the given unfolding.

The longer path in the figure, a 4-face \LS-path, is obtained by edge-rolling the cube right, up, and right (RUR).  The notations ``bottom,'' ``right,'' ``back,'' and ``top'' on the grid squares indicate which of the original cube faces contact those  squares.  The shorter path, a 3-face \LS-path, is obtained by rolling the cube left two times (LL), and the grid faces are analogously designated.   

The shaded square in Figure~\ref{f:ur1} together with the three squares to its right are a partial unfolding, as is the shaded square together with the two squares to its left.  The entire diagram is not an unfolding, even though it has six squares, which is the required number of cube faces.  The entire diagram is also not the full unfolding grid, which appears later in Figure~\ref{f:unfgrid}.  When we have that grid, we use it to find all \LS-paths, and the shortest of these is (or are) the shortest paths. (This relies on the assumption, mentioned in the section on applications of corner moves and proved in Appendix~A, that the shortest path must be one of the \LS-paths, i.e. must be a straight line in the interior of at least one of the partial unfoldings.)  

Figure~\ref{f:ur1} also contains a square labeled ``top view''.  This is a view of the cube looking down from the top onto the target face, with the target point~$ T $ projected downwards and so also appearing on the source face along with the source point~$ S $.  This is a compact way to keep track of source and target locations, and  we have incorporated it into the (shaded) base face of the unfolding grid.	

Warning: rolling and unfolding are not fully equivalent in this diagram, and it is the unfoldings we care about.  The failure of equivalence comes from the fact that, for instance, rolling right and then up is not the same as rolling up and then right (so $ \text{RU} \ne \text{UR} $).  The cube comes to rest on the same square in the unfolding grid either way, but the cube face in contact with the square is the back face if we do right-up but the right face if we do up-right. See Appendix~A for further commentary on the interpretation of pseudopaths in the unfolding grid.

\section{Combinatorics of unfoldings}
Our results come from a combination of numerical calculations and conceptual principles that simplify and guide the numerical work.  This section is about the conceptual part, which obtains relationships among various unfoldings by means of modifications we call \emph{corner moves}.  A corner move is a transformation taking one unfolding to another.  It is associated with a vertex of the cube that has exactly one edge cut in the unfolding. The corner move  corresponds to changing which of the three edges at the vertex is cut.  Figure~\ref{f:cm1} illustrates, at the top, the unfoldings resulting from cutting one of three edges at a vertex and, underneath, the corner moves transforming one unfolding to the other---any order of following arrows to the same target results in the same unfolding.
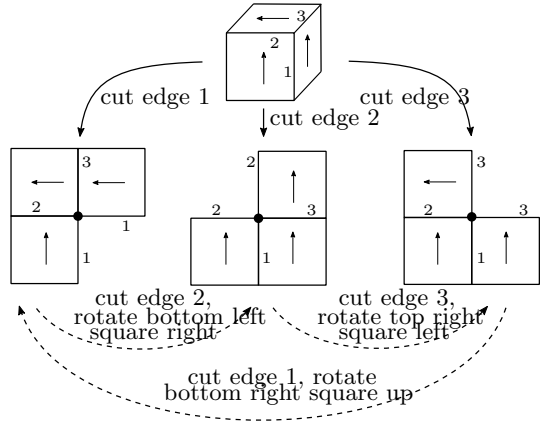
\begin{figure}[h]
	\centering{
	\begin{overpic}[scale=0.12]{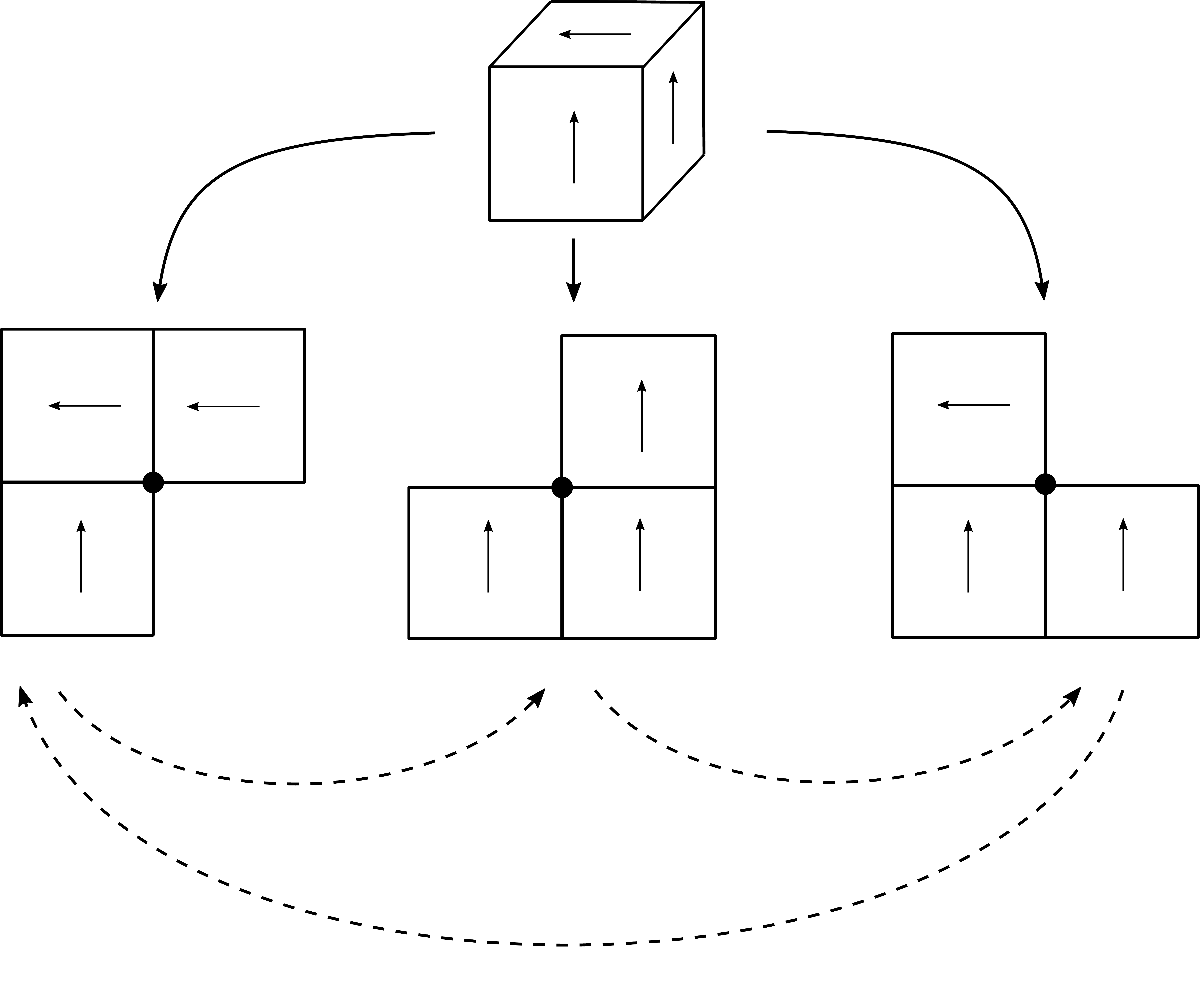} %\put(){\tiny  }
		\put(51.5,70){\tiny 1}
		\put(49,75.5){\tiny 2}
		\put(54,80.7){\tiny 3}
		%--------------------
		\put(17,65){\footnotesize cut edge 1}
		\put(49,61){\footnotesize cut  edge 2}
		\put(66,65){\footnotesize cut  edge 3}
		%--------------------
		\put(16,27){\footnotesize cut edge 2,}
		\put(12,23.5){\footnotesize rotate bottom left}
		\put(15,20.5){\footnotesize square right}
		%--------------------
		\put(62,27){\footnotesize cut edge 3,}
		\put(58,23.5){\footnotesize rotate top right}
		\put(62,20.5){\footnotesize square left}
		%--------------------
		\put(33.5,12){\footnotesize cut edge 1, rotate }
		\put(29,9){\footnotesize bottom right square up}
		%---------------------
		\put(13.5,35){\tiny 1}
		\put(21,41){\tiny 1}
		\put(4,44.5){\tiny 2}
		\put(13.5,53){\tiny 3}
		%---------------------
		\put(47.5,35){\tiny 1}
		\put(56,44){\tiny 3}
		\put(38,44){\tiny 2}
		\put(44.5,53){\tiny 2}
		%---------------------
		\put(88,35){\tiny 1}
		\put(78,44){\tiny 2}
		\put(96,44){\tiny 3}
		\put(88,53){\tiny 3}
	\end{overpic}
    }
\caption{Corner Moves \label{f:cm1}}
\end{figure}

Corner moves can be used to generate all the unfoldings of the cube as illustrated in Figure~\ref{f:11u}. Pairs of opposite faces are marked, and the single arrows indicate a corner move obtained by rotating a single square around a pivot vertex.  The one set of double arrows denotes a corner move that rotates a three-square group around a pivot vertex.
\begin{figure}[h]
	\centering{
		\begin{overpic}[scale=0.35]{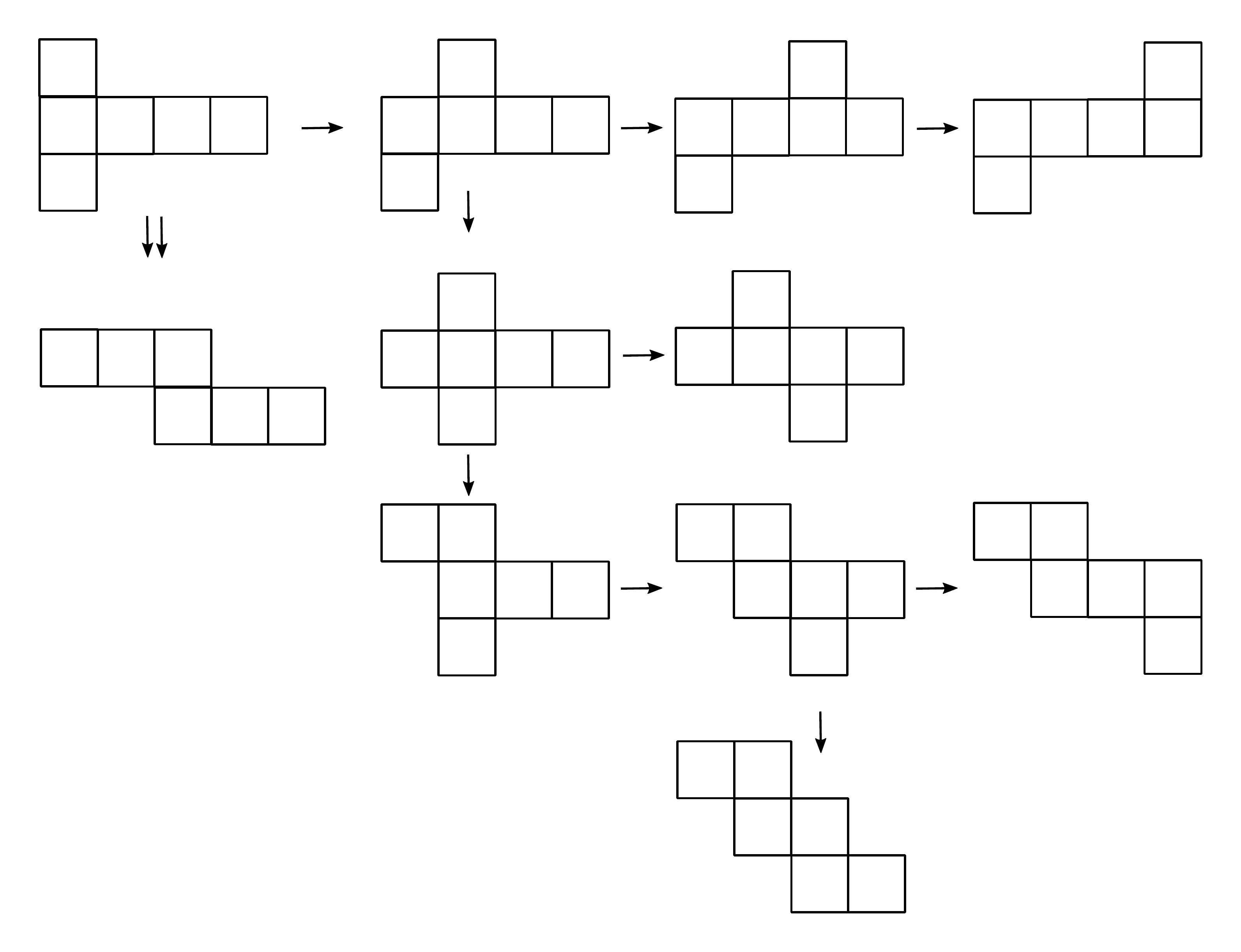} %\put(){\footnotesize $ \spadesuit  $ }
			\put(4.5,70){\footnotesize $ \heartsuit $ }
			\put(4.5,61){\footnotesize $ \heartsuit $ }
			\put(4.5,65.5){\footnotesize $ \spadesuit  $ }
			\put(13.5,65.5){\footnotesize $ \spadesuit  $ }
			%-------------------------------------------
			\put(36.5,72){\footnotesize  \rotatebox{-90}{$\heartsuit$} }
			\put(32,61){\footnotesize $ \heartsuit $ }
			\put(32,65.5){\footnotesize $ \spadesuit  $ }
			\put(41.2,65.5){\footnotesize $ \spadesuit  $ }
			%-------------------------------------------
			\put(64.8,72){\footnotesize \rotatebox{180}{$\heartsuit$} } 
			\put(55.7,61){\footnotesize $ \heartsuit $ }
			\put(55.7,65.5){\footnotesize $ \spadesuit  $ }
			\put(64.8,65.5){\footnotesize $ \spadesuit  $ }
			%-------------------------------------------
			\put(94,70){\footnotesize \rotatebox{90}{$\heartsuit$} }
			\put(79.8,61){\footnotesize $ \heartsuit $ }
			\put(79.8,65.5){\footnotesize $ \spadesuit  $ }
			\put(89,65.5){\footnotesize $ \spadesuit  $ }
			%--------------------------------------------------------
			%--------------------------------------------------------
			\put(4.5,48.5){\footnotesize \rotatebox{-90}{$\heartsuit$} }
			\put(13.7,48.5){\footnotesize \rotatebox{-90}{$\heartsuit$} }
			\put(9,48.5){\footnotesize \rotatebox{-90}{$\spadesuit$} }
			\put(18.3,42){\footnotesize $ \spadesuit  $ }
			%-------------------------------------------
			\put(36.5,51.3){\footnotesize $ \heartsuit $ }
			\put(36.5,42){\footnotesize \rotatebox{90}{$\heartsuit$} }
			\put(41.2,47){\footnotesize $ \spadesuit $ }
			\put(32,47){\footnotesize $ \spadesuit $ }
			%-------------------------------------------
			\put(60.3,51.3){\footnotesize $ \heartsuit $ }
			\put(64.8,44){\footnotesize \rotatebox{180}{$\heartsuit$} }
			\put(55.7,47){\footnotesize $ \spadesuit $ }
			\put(64.8,47){\footnotesize $ \spadesuit $ }
			%---------------------------------------------------------
			%---------------------------------------------------------
			\put(36.5,32.5){\footnotesize $ \heartsuit $ }
			\put(36.5,23.5){\footnotesize \rotatebox{90}{$\heartsuit$} }
			\put(32,34.5){\footnotesize \rotatebox{-90}{$\spadesuit$} }
			\put(41.2,28){\footnotesize $ \spadesuit $ }
			%-------------------------------------------
			\put(60.3,32.5){\footnotesize $ \heartsuit $ }
			\put(64.9,25.5){\footnotesize \rotatebox{180}{$\heartsuit$} }
			\put(55.7,34.5){\footnotesize \rotatebox{-90}{$\spadesuit$} }
			\put(64.9,28){\footnotesize $ \spadesuit $ }
			%-------------------------------------------
			\put(84.2,32.5){\footnotesize $ \heartsuit $ }
			\put(93.5,25.5){\footnotesize \rotatebox{-90}{$\heartsuit$} }
			\put(80,34.5){\footnotesize \rotatebox{-90}{$\spadesuit$} }
			\put(88.8,28){\footnotesize $ \spadesuit $ }
			%---------------------------------------------------------
			%---------------------------------------------------------
			\put(60.3,13.5){\footnotesize $ \heartsuit $ }
			\put(64.9,6.4){\footnotesize \rotatebox{180}{$\heartsuit$} }
			\put(55.7,13.5){\footnotesize $ \spadesuit $ }
			\put(64.9,9){\footnotesize $ \spadesuit $ }
		\end{overpic}	
	}
	\caption{The eleven congurence classes of unfoldings via corner moves \label{f:11u}}
\end{figure}
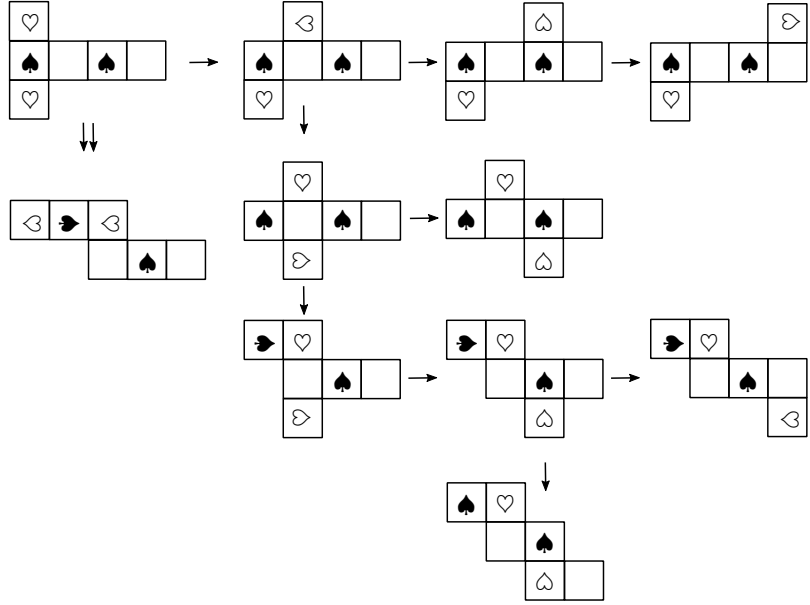 
Eleven incongruent shapes are depicted, and it has been known for a long time that these are all possible incongruent shapes.  An elementary account of this fact is in The College Mathematics Journal,~\cite{rGrSV}. 

Inspection of these possibilities indicates that, up to symmetry, there is a unique unfolded configuration for each type of \LS-path between opposite faces, with ``type'' referring to how many faces are traversed.  These configurations are shown in Figure~\ref{f:fapaty}.
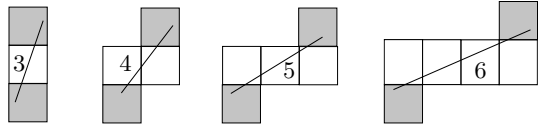
\begin{figure}[h]
	\centering{
		\begin{overpic}[scale=0.45]{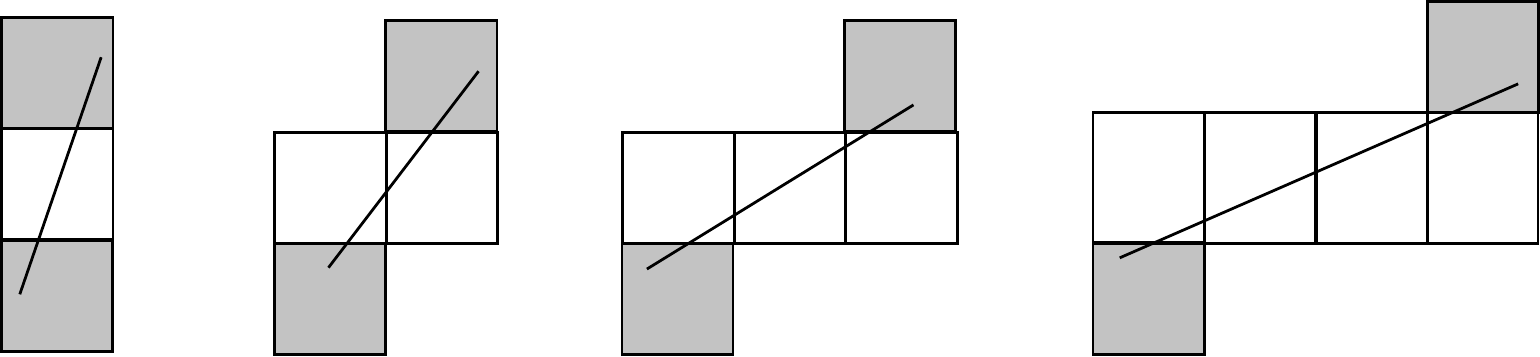} %\put(){\footnotesize }
			\put(0.9,10){\footnotesize 3}
			\put(21,10){\footnotesize 4}
			\put(51.9,8.5){\footnotesize 5}
			\put(88,8.5){\footnotesize 6}
		\end{overpic}
	}
	\caption{\label{f:fapaty} Generic \LS-path configurations}
\end{figure}
The numbers next to the sample paths indicate how many faces have been traversed.

\section{Euclidean geometry of corner moves}
Corner moves give us a way to compare the lengths of straight-line paths in related unfoldings.  The details are in Figure~\ref{f:cmg}.  Begin with an unfolding containing the leftmost square, the three shaded squares, the source point~$ S $, and a target point image~$ T_1 $ with the \LS-path of length~$ d_1 $.  (The ``torn off'' piece indicates a possible location where additional unfolding squares could be bound.) 
\begin{figure}[h]
	\centering{
		\begin{overpic}[scale=0.12]{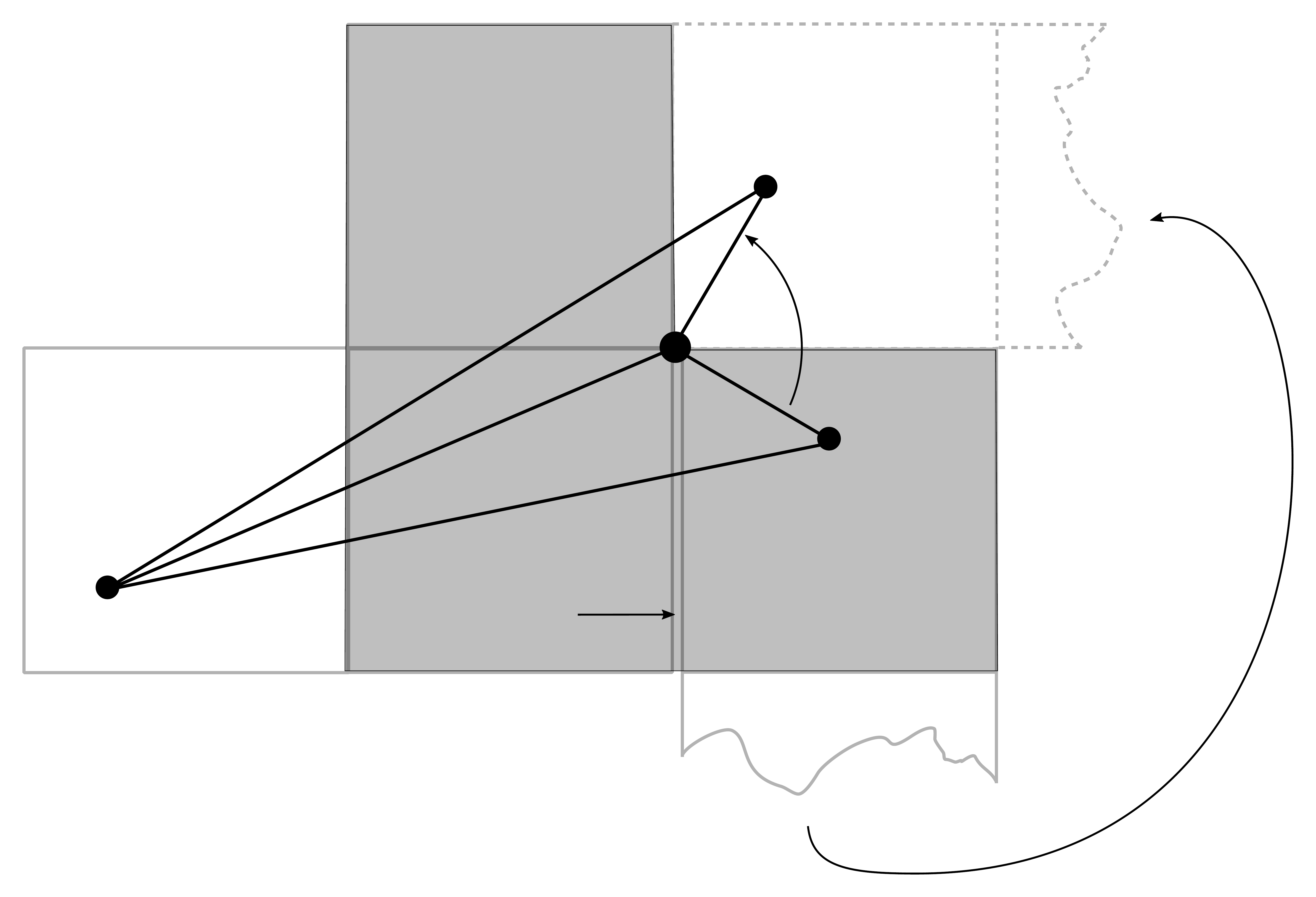} %\put(){\footnotesize }
			\put(53,42){\footnotesize $ P $ }
			\put(4.5,22){\footnotesize $ S $ }
			\put(65,33){\footnotesize $ T_1 $ }
			\put(60,55){\footnotesize $ T_2 $ }
			\put(38,26){\footnotesize $ d_1 $ }
			\put(38,45){\footnotesize $ d_2 $ }
			\put(29,20.5){\footnotesize edge cut }
			\put(77,31){\footnotesize pivot lower}
			\put(77,27){\footnotesize right square}
			\put(77,23){\footnotesize $ 90^\circ $ around}
			\put(77,19){\footnotesize $ P $}
			\put(61.5,44){\footnotesize $ 90^\circ $}
		\end{overpic}	
	}
	\caption{Corner move geometry \label{f:cmg}}
\end{figure} 
Using the indicated edge cut, implement the corner move that rotates the lower right-hand shaded square up around~$ P $. This rotation produces a new unfolding with a new target point image~$ T_2 $ and a corresponding \LS-path of length~$ d_2 $.
 
Since we are looking for the shortest path, we'd like to know how~$ d_2 $ compares to~$ d_1 $.  The applicable geometry is the extended side-angle-side theorem (sometimes called the \emph{hinge theorem}), which says that if two sides of one triangle are congruent to two sides of another triangle, then the size relation between the included angle of the first triangle  and the included angle of the second triangle is the same as the size relation between the third side of the first triangle and the third side of the second triangle.   Applied to the diagram in Figure~\ref{f:cmg}, the SAS theorem says that the size relation between~$ d_1 $ and~$ d_2 $ is the same as the size relation between the angles~$ \angle SPT_1 $ and~$\angle SPT_2 $.  In view of the indicated~$ 90^\circ $ angle, we have $ d_1=d_2 $ if and only if $ \angle SPT_1 = 135^\circ =  \angle SPT_2 $.  Fixing our attention on the original unfolding, we can say that \emph{the indicated corner move will result in a shorter path if and only if $ \angle SPT_1 > 135^\circ $}, and we call~$ \angle SPT_1 $ the \emph{decision angle} for the original path.

There is a problem with corner moves: they may result in pseudopaths---see Figure~\ref{f:pseudp}.  The pseudopaths that concern us are ones obtainable from the examples in Figure~\ref{f:fapaty} by adjusting the illustrated path until some portion of it---but not the endpoints in source and target faces---lies outside the unfolding.
\begin{figure}[h]
	\centering{
		\begin{overpic}[scale=0.3]{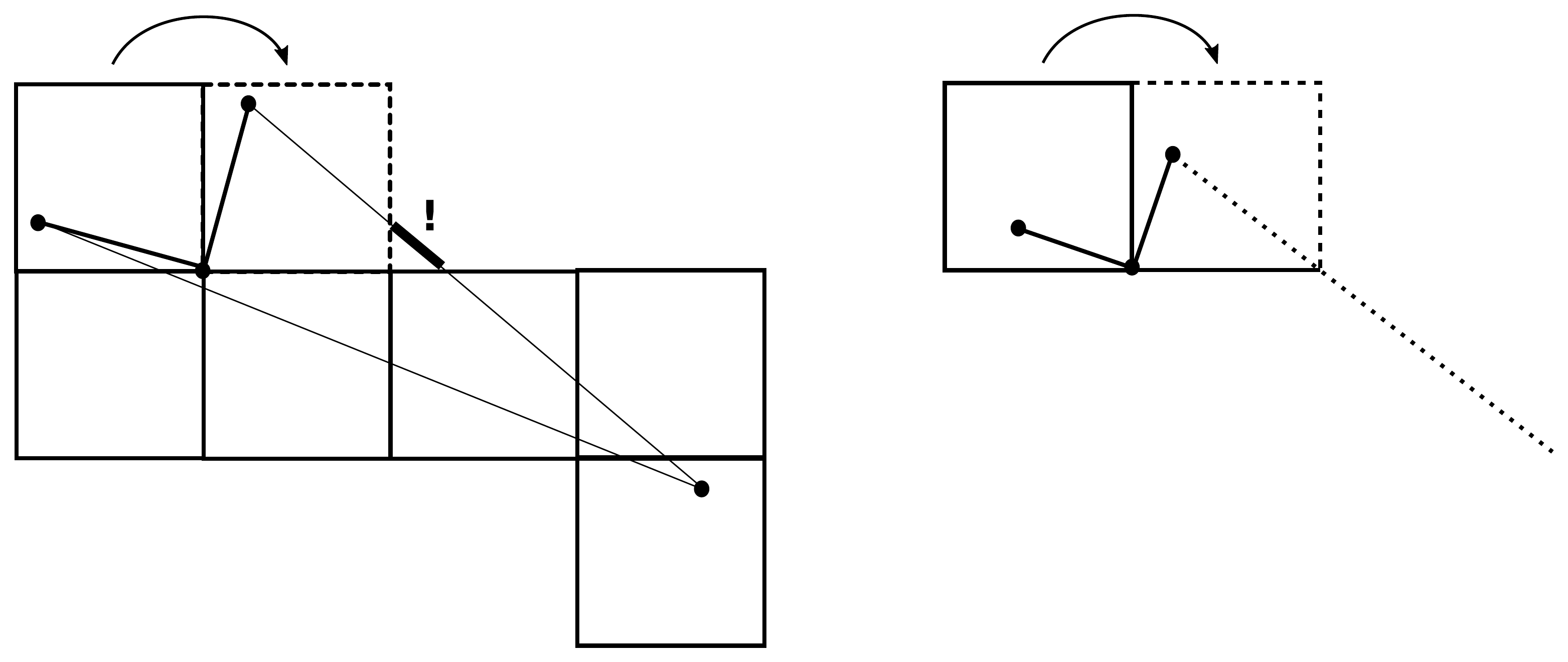} % \put(){\footnotesize }
			\put(2.5,29){\footnotesize $ T_1 $ }
			\put(19,33.5){\footnotesize $ T_2 $ }
			\put(45,8){\footnotesize $ S $ }
			%----------------------------------------
			\put(63,29){\footnotesize $ T_1 $ }
			\put(73,34){\footnotesize $ T_2 $ }
			\put(71,20){\footnotesize path source}
			\put(71,17){\footnotesize locations}
			
			\put(93,25){\footnotesize pseudopath}
			\put(93,22){\footnotesize  source}
			\put(93,19){\footnotesize  locations}
			%----------------------------------------
			\put(23,0){\footnotesize (a) }
			\put(80,0){\footnotesize (b) }
		\end{overpic}
	}
	\caption{\label{f:pseudp}Corner moves resulting in pseudopaths}
\end{figure}
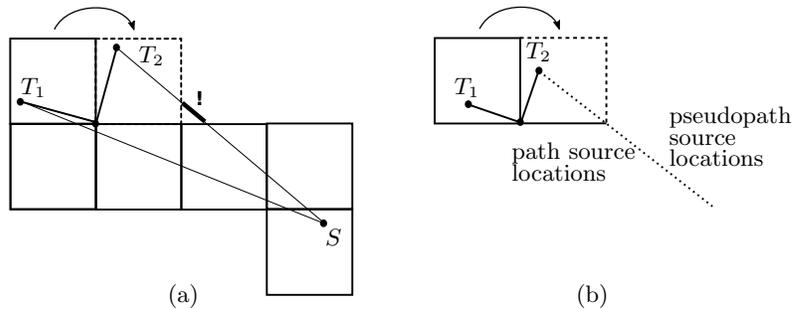
 We have been willing to name pseudopaths, rather than ban corner moves that produce them, because pseudopaths are useful---see the section on applications of corner move geometry, where pseudopaths play an essential role.

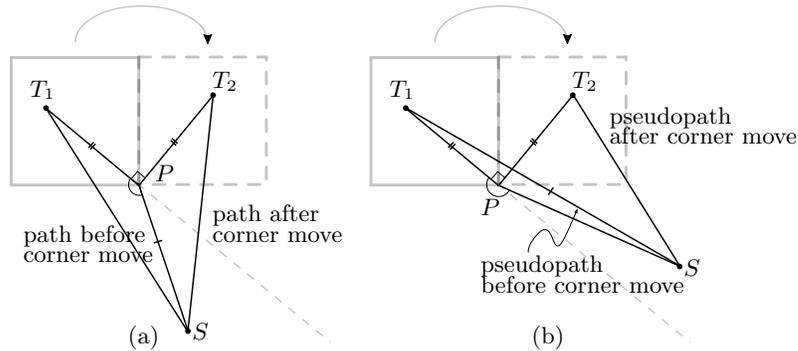
\begin{figure}[h]
	\centering{
		\begin{overpic}[scale=0.35]{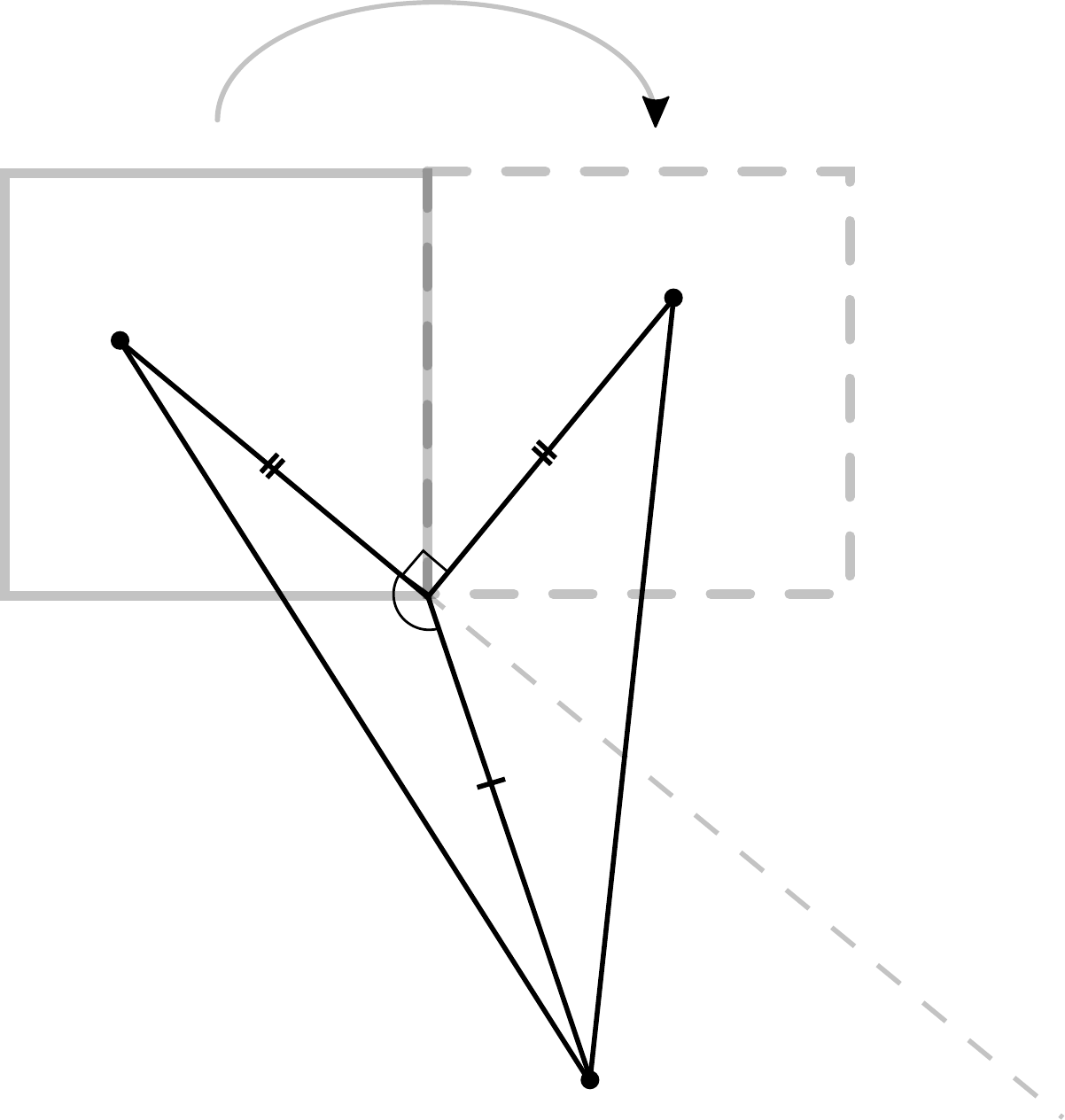} %\put(){\footnotesize }
			\put(54,1){\footnotesize $S$}
			\put(6,73){\footnotesize $T_1$}
			\put(60,75.5){\footnotesize $T_2$}
			\put(43,48){\footnotesize $P$}
			\put(4,30){\footnotesize path before}
			\put(4,24){\footnotesize corner move}
			\put(60,36){\footnotesize path after}
			\put(60,30){\footnotesize corner move}
			\put(35,0){\footnotesize (a)}
		\end{overpic}
		\quad
		\begin{overpic}[scale=0.35]{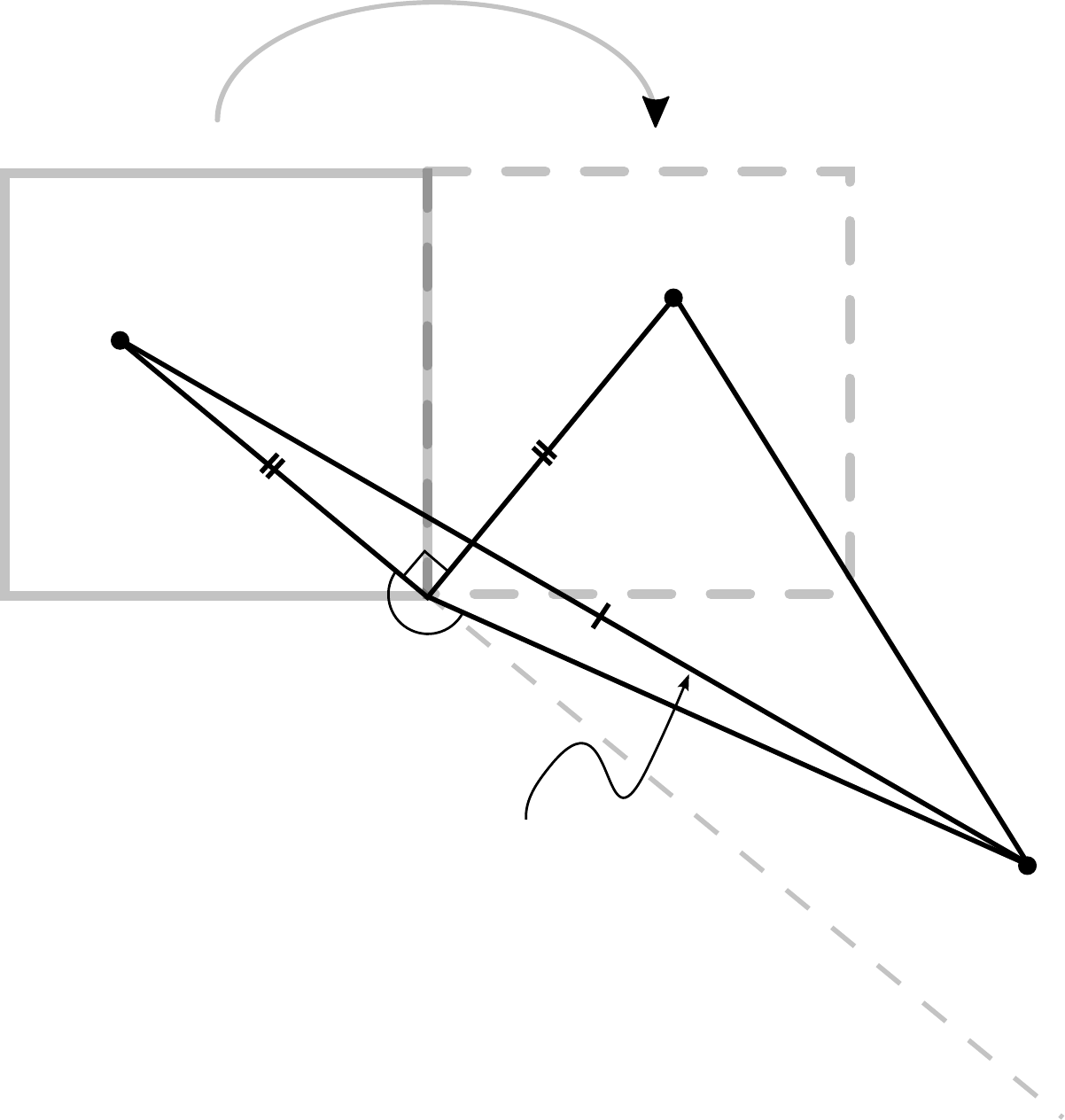} %\put(){\footnotesize }
			\put(93,20){\footnotesize $S$}
			\put(6,73){\footnotesize $T_1$}
			\put(60,75.5){\footnotesize $T_2$}
			\put(33,38){\footnotesize $P$}
			\put(33,21){\footnotesize pseudopath }
			\put(33,15){\footnotesize before corner move}
			\put(71,65){\footnotesize pseudopath}
			\put(71,59){\footnotesize  after corner move}
			\put(48,0){\footnotesize (b)}
		\end{overpic}
	}
	\caption{\label{f:cmgeom} Geometry of paths and pseudopaths}
\end{figure}

In Figure~\ref{f:cmgeom}\,(a), a corner move takes an \LS-path to an \LS-path.  In diagram~(b), we have a corner move applied to a pseudopath and resulting in another pseudopath.   The SAS argument for the latter situation compares the angles~$ \angle SPT_1 $ and~$ \angle SPT_2 $, and in this case $ \angle SPT_2 + 90^\circ = \angle SPT_1 $, making $ \angle SPT_2 < \angle SPT_1 $, and so the length of~$ ST_2 $ is always less than the length of~$ ST_1 $.  This argument is generic: a pseudopath can always be shortened by a corner move, with the result either a shorter \LS-path or a shorter pseudopath.  But if the result is a shorter pseudopath, then another corner move will shorten that.  We end up with a sequence of corner moves and successively shorter paths that continues as long as a pseudopath is the result.  This cannot go on for ever; if we eventually produce a 3-face path it will have to be an \LS-path.  We record these results in the following proposition.

\begin{prop} \label{p:decang}
	If the decision angle for an \LS-path or pseudopath exceeds~$135^\circ$, the path can be replaced by a shorter path (\LS\ or pseudo) by applying a corner move with the decision angle vertex as its pivot point.  If we begin with a pseudopath or if a corner move results in a pseudopath, then there is a sequence of corner moves that must eventually end with a shorter \LS-path.
\end{prop}

\begin{exer}
	The SAS argument used for the configuration depicted in Figure~\ref{f:cmgeom}\,(b) doesn't work if $  \angle SPT_1 $ is~$ 180^\circ $, because one of the two triangles for SAS no longer exists.  But the conclusion that the corner move will produce a shorter path is still true.  The general situation is depicted in Figure~\ref{f:5fp}\,(a);  use this to replace the SAS argument.
\end{exer}

\section{Three applications of corner move geometry}

\subsection{Shortest paths have to be straight lines contained in the interior of some unfolding.}
This is often taken for granted, for example by Dudeney and Mellinger-Viglione. We provide an argument based on corner moves in Appendix~A.

\subsection{Non-existence of 5- and 6-face shortest paths.}  
Here is the argument for 5-face \LS-paths.  Up to a symmetry-equivalent diagram, we have the third configuration in Figure~\ref{f:fapaty}, reproduced in Figure~\ref{f:5fp}\,(b) with some extra details.  A 5-face \LS-path must have its  endpoints in the shaded regions. The dotted line from a possible source~$ S $ to the pivot vertex~$ P $ to a possible target~$ T_1 $ defines the decision angle for~$ ST_1 $.   All such decision angles are greater than~$ 135^\circ $, as is made evident by the black-filled~$ 135^\circ $ angle, so a corner move can be applied to obtain a shorter 4-face \LS-path or pseudopath.  If we get an \LS-path, we are done.  If we get a pseudopath, then a second corner move will produce a 3-face \LS-path and we are done in this case too.
\begin{figure}[h]
	\centering{
		\begin{overpic}[scale=0.3,trim=0cm -1cm 0cm 0cm,clip]{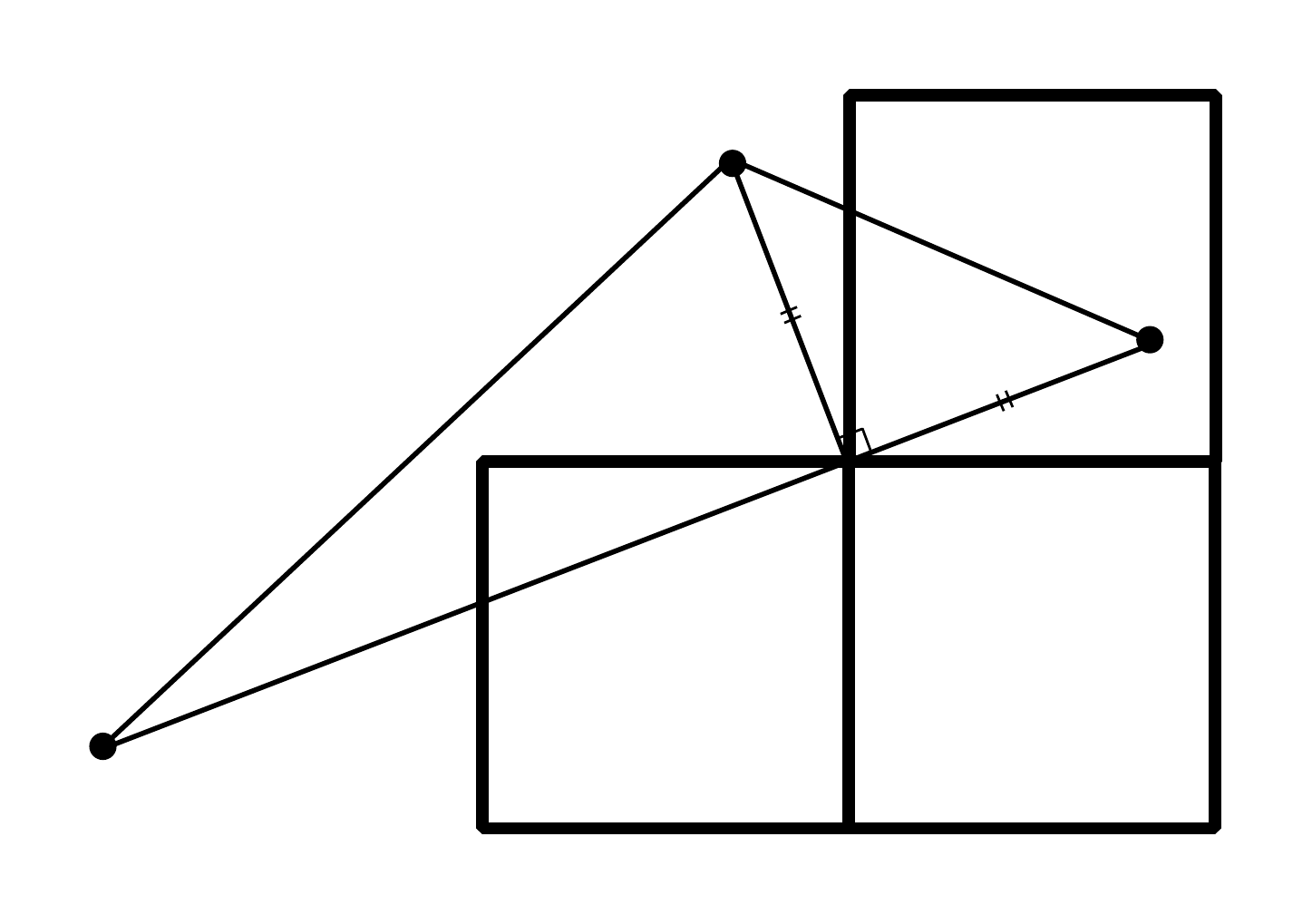} \put(4,13){\footnotesize $ S $}
		\put(66,36){\footnotesize $ P $}
		\put(85,44){\footnotesize $ T_1 $}
		\put(54,67){\footnotesize $ T_2 $}
		\put(56,6){\footnotesize (a)}
		\end{overpic}
	\qquad
		\begin{overpic}[scale=0.7]{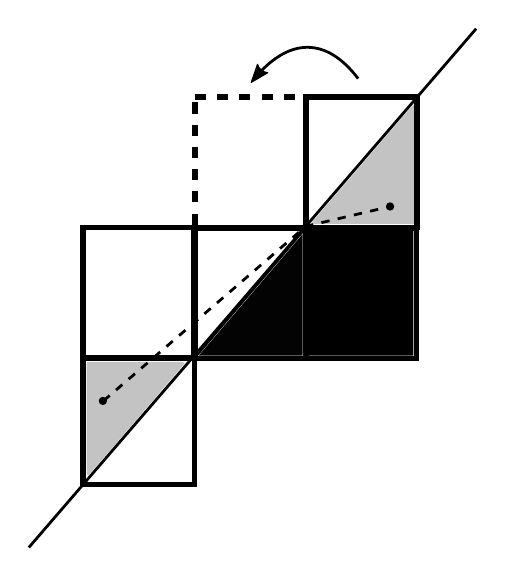} %\put(){\footnotesize }
			\put(45,64){\footnotesize $P$}
			\put(36,29){\footnotesize $135^\circ$ angle}
			\put(35,93.5){\footnotesize corner move}
			\put(-12,29){\footnotesize source $ S $}
			\put(75,63){\footnotesize target $ T_1 $}
			\put(49,9){\footnotesize (b)}
		\end{overpic}
	}
	\caption{\label{f:5fp} Applications of corner move geometry}
\end{figure}
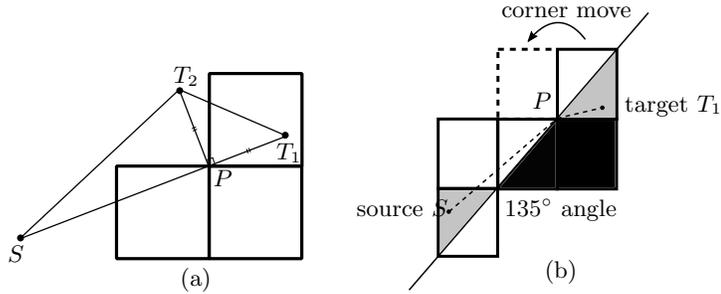

\begin{exer}
Show that there are no 6-face shortest paths by modifying the diagram in Figure~\ref{f:5fp}\,(b). 
\end{exer}

We have arrived at a result which, by eliminating the 5-face shortest path possibility,  distinguishes the cube from the rectangular solids considered by Dudeney and Mellinger-Viglione.

\begin{thm} \label{t:3or4faces}
	A shortest path on a cube between source and target points in the interiors of opposite faces either traverses three or four faces.
\end{thm}

\subsection{A centroid source or  target point always has a 3-face shortest path.}  
The arguments for source and target points are identical. For centroid targets, Figure~\ref{f:no4} represents any one of the four 3-face path roll-outs RR, UU, LL, and DD in the right-hand image in Figure~\ref{f:unfgrid}.  Consequently, regardless of the source point location, the decision angles for any 3-face \LS-path to a centroid target point are always less than~$ 135^\circ $. It follows that no corner move can shorten a 3-face \LS-path, and so the shortest path must be one of the four 3-face \LS-paths.
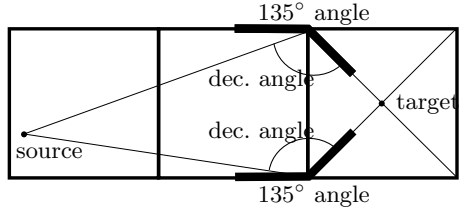
\begin{figure}[ht]
\centering{
	\begin{overpic}[scale=0.4]{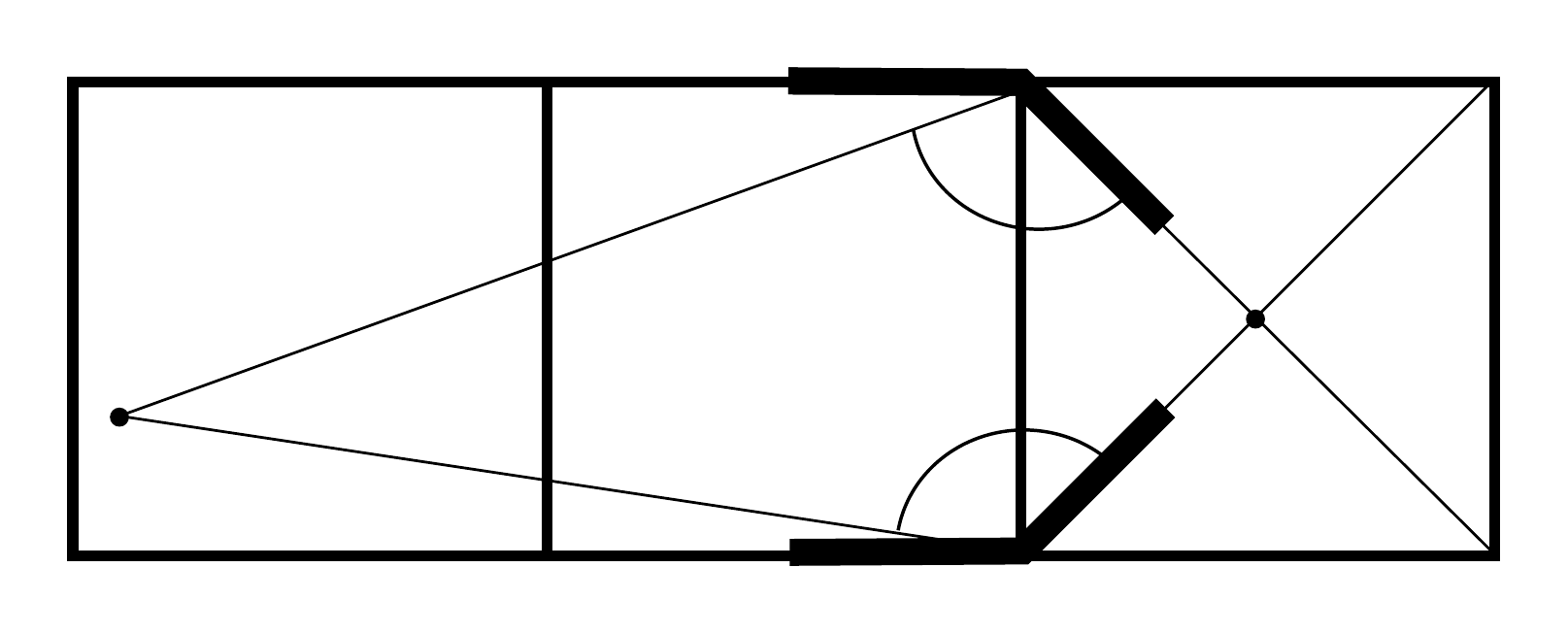} %\put(){} change source, target to S, T_1
		\put(6,9){\footnotesize source}
		\put(83,19){\footnotesize target}
		\put(55,37){\footnotesize $ 135^\circ $ angle}
		\put(55,0){\footnotesize $ 135^\circ $ angle}
		\put(45,23.5 ){\footnotesize dec.~angle}
		\put(45,13){\footnotesize dec.~angle}
	\end{overpic}
}
\caption{\label{f:no4}No 4-face shortest path for centroid target}
\end{figure}

\section{Analytic geometry of corner moves}
In this section, we introduce coordinates and obtain formulas for the points~$ T_1 $ and~$ T_2 $ and the distances~$ d_1 $ and~$ d_2 $  in Figure~\ref{f:cmg} in terms of the positions of~$ S= (s_1,s_2) $ and the projected target point on the source face,~$ T= (t_1,t_2) $.  The context for coordinatization is the unfolding grid described earlier and partially depicted in Figure~\ref{f:ur1}.  The full grid, with sample source and target points and \LS-paths and pseudopaths, is illustrated in Figure~\ref{f:unfgrid}. In view of the proposition above, the grid only has to accomodate \LS-paths and pseudopaths traversing three or four faces. Pseudopaths are denoted by dotted lines.
\begin{figure}[ht]
	\centering{
\begin{overpic}[scale=0.15]{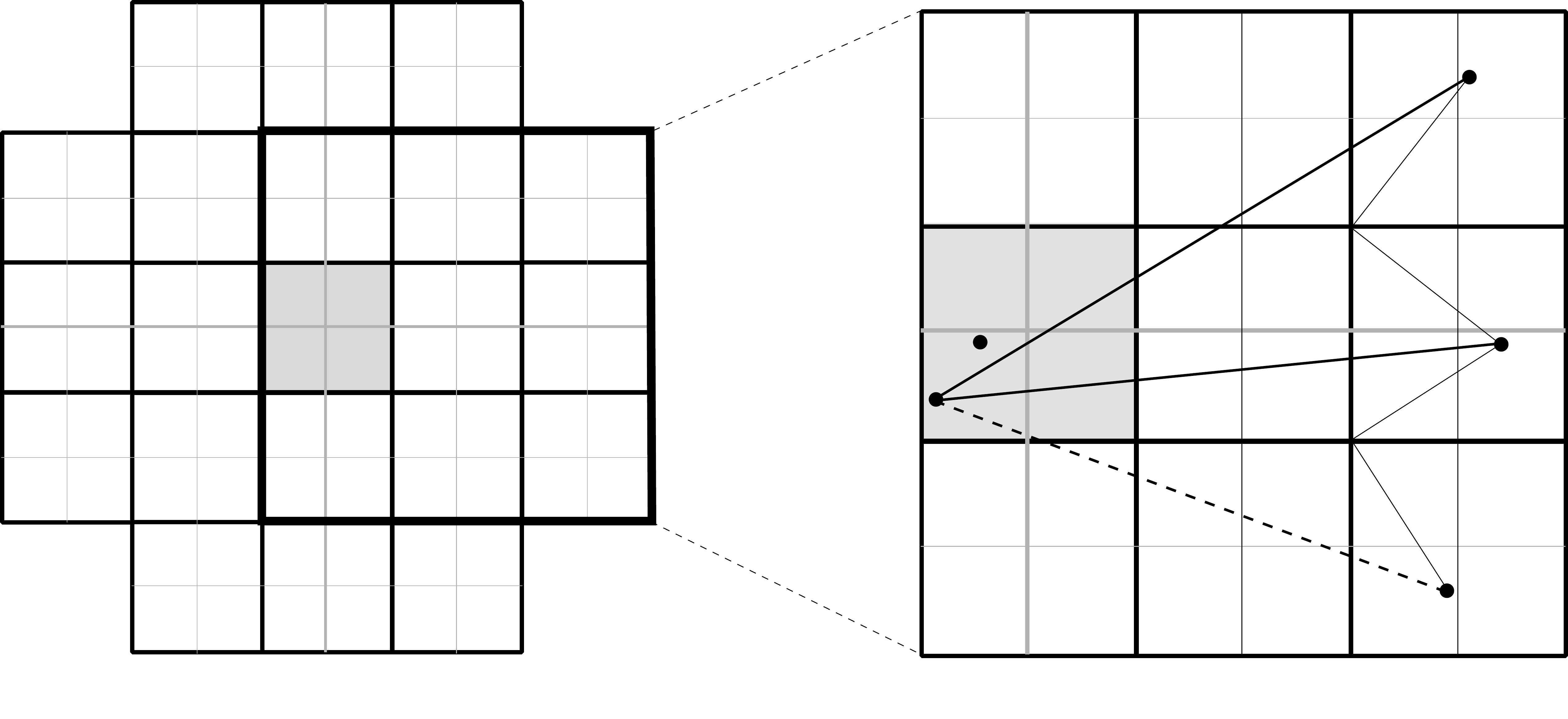} %\put(){\tiny $ $ }
	\put(57,1 ){\tiny $ -1 $}
	\put(65.1,1 ){\tiny $ 0 $}
	\put(72,1 ){\tiny $ 1 $}
	\put(78.6,1 ){\tiny $ \mathbf{2}$}
	\put(85.7,1 ){\tiny $ 3 $}
	\put(92.5,1 ){\tiny $ 4 $}
	\put(99.5,1 ){\tiny $ 5 $}
	%-----------------------------------------------
	\put(53,3.5 ){\tiny $ -3 $}	
	\put(53,10.4 ){\tiny $ -2 $}
	\put(53,17 ){\tiny $ -1 $}
	\put(53,24.3 ){\tiny\hphantom{$-$}$ 0 $}
	\put(53,31 ){\tiny\hphantom{$-$}$ 1 $}	
	\put(53,37.7 ){\tiny\hphantom{$-$}$ 2 $}
	\put(53,44.3 ){\tiny\hphantom{$-$}$ 3 $}
	%------------------------------------------------
	\put(56,20){\tiny $ S $}
	\put(61.9,25){\tiny $ T $}
	\put(95.5,20.5){\tiny $ T_1 $}
	\put(94,41.5){\tiny $ T_2 $}
	\put(93.5,6.5){\tiny $ T_3 $}
	%------------------------------------------------
	\put(77.5,14){\tiny RDR }
	\put(76,23){\tiny RR }
	\put(75,34){\tiny RUR }
\end{overpic}
\quad
\begin{overpic}[scale=0.22]{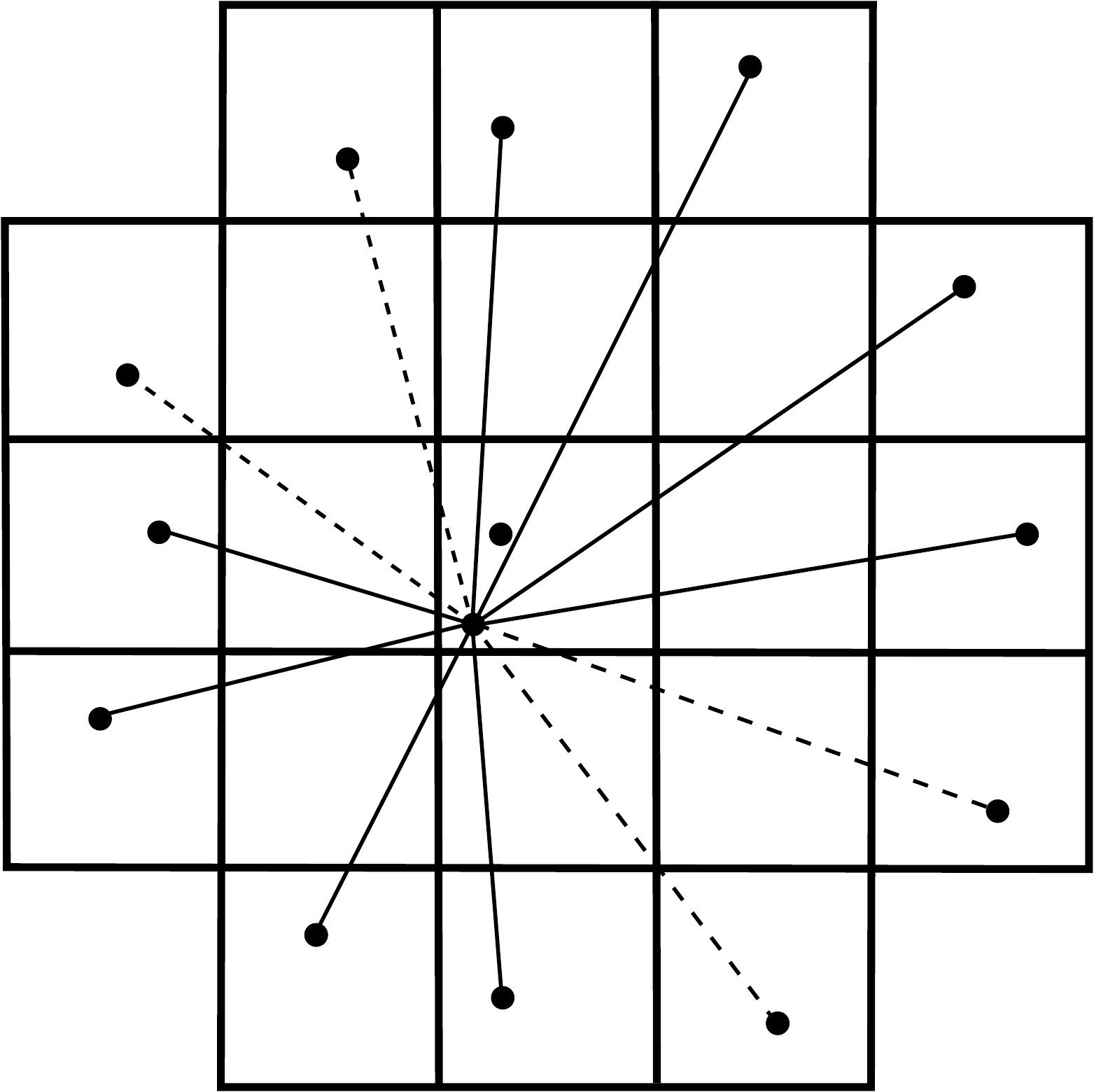} %\put(){\tiny  }
	\put(44,34){\tiny $S $ }
	\put(44,53){\tiny $T $ }
	\put(89,43){\tiny RR }
	\put(42,91){\tiny UU }
	\put(3,48){\tiny LL }
	\put(42,2){\tiny DD }
	\put(84,65){\tiny RUR }
	\put(84,30){\tiny RDR }
	\put(65,81){\tiny URU }
	\put(22,88){\tiny ULU }
	\put(3,68){\tiny LUL }
	\put(3,25){\tiny LDL }
	\put(22,7){\tiny DLD }
	\put(64,12){\tiny DRD }
\end{overpic}
}
\caption{\label{f:unfgrid} The unfolding grid, corner moves, and all paths (pseudopaths dotted)}
\end{figure}
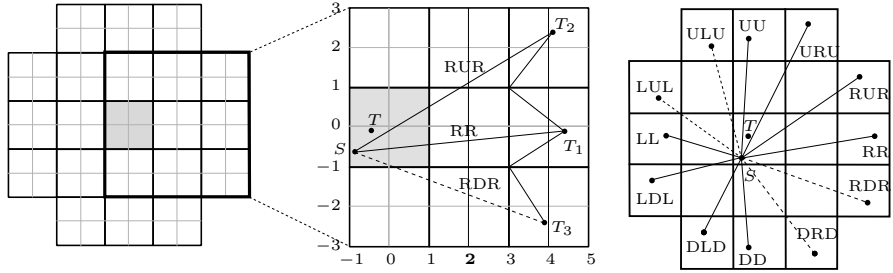

Recall that in the unfolding grid, the source and target points are represented as in the ``top view'' illustrated in Figure~\ref{f:ur1} and described in the text following that figure. These points appear together in the base face of the unfolding grid, which is shaded in the left and middle diagrams of Figure~\ref{f:unfgrid}.  In the middle diagram of Figure~\ref{f:unfgrid}, the image point~$ T_1 $, obtained via the RR rolling sequence, can instead be plotted by reflecting~$ T $ in the line~$ x=2 $. The corner-move images~$ T_2 $ and~$ T_3 $ are then obtained from~$ T_1 $ by rotating around the pivot vertices~$ (3,1) $ and~$ (3,-1) $ respectively.  The resulting \LS-paths and pseudopath are labeled in the diagram by their roll sequences.

In addition to the paths, \LS\ and pseudo, shown in the middle diagram, there are analogous 3-face roll sequences to the left, up, and down, whose results can be plotted by reflection in the lines~$ x=-2 $, $ y=2 $, and~$ y=-2 $ respectively. The right-hand diagram of Figure~\ref{f:unfgrid} is an example of all twelve target positions, with the LL path obviously the shortest path.  The third column of Table~\ref{ta:LSPends}, in the rows for~RR, LL, UU, and~DD, gives the coordinates of the resulting 3-face target images for the original target point~$ T= (t_1,t_2) $. 

In order to get the 4-face path endpoints, we have to rotate the endpoints~$ (x,y) $ of the 3-face \LS-paths about the appropriate pivot points~$ (a,b) $.  The general formulas for these rotations are in the left-hand column of the following diagram.
{%\small
\[
\xymatrix{ % @R-1pc compresses rows if needed	
 (b+a-y,\, b-a+x) 
 &&&& 
 (b-y,\, x-a) \ar@{|->}[llll]_-{\text{translate back}}
 \\
 (x,y) \ar@{|->}[rrrr]^-{\text{translate rotation center to origin}} 
       \ar@{|->}[u]^{\text{$+90^\circ$ around $(a,b)$} } \ar@{|->}[d]_{\text{$-90^\circ$ around $(a,b)$} }
 &&&& 
 (x-a,\, y-b) \ar@{|->}[d]^{\text{rotate $ -90^\circ $}} \ar@{|->}[u]_{\text{rotate $ +90^\circ $}}
 \\
 (a-b+y,\, a+b-x) 
 &&&& 
 (y-b,\, a-x) \ar@{|->}[llll]^-{\text{translate back}}
}
\]
}
The rotation formulas are used for the entries in the final column of Table~\ref{ta:LSPends}. 
\begin{table}[ht]
	%\scriptsize
\centering{	
	\begin{tabular}{ccccc} \toprule
		 Roll Seq. & Refl. Line & Coords. $ (x,y) $ & Piv. Pt. $ (a,b) $ &  Rot. Pt. Coords.\\ \midrule
		 RR  &  $ x=\hphantom{-}2 $ & $ (4-t_1,\, t_2) $ \\ 
		 RUR &&& $ (3,1) $          & $ (4-t_2,\, 2-t_1) $ \\ 
		 RDR &&& $ (3,-1) $         & $ (4+t_2,\, -2+t_1) $ \\ \midrule
		 LL  &  $ x=-2 $            & $ (-4-t_1,\, t_2) $ \\ 
		 LUL &&& $ (-3,1) $         & $ (-4+t_2,\, 2+t_1) $\\ 
		 LDL &&& $ (-3,-1) $        & $ (-4-t_2,\, -2-t_1) $ \\ \midrule
		 UU  &  $ y=\hphantom{-}2 $ & $ (t_1,\, 4-t_2) $ \\
		 URU &&& $ (1,3) $          & $ (2-t_2,\, 4-t_1) $ \\ 
		 ULU &&& $ (-1,3) $         & $ (-2+t_2,\, 4+t_1) $ \\ \midrule 
	     DD  &  $ y=-2 $            & $ (t_1,\, -4-t_2) $ \\
		 DRD &&& $ (1,-3) $         & $ (2+t_2,\, -4+t_1) $ \\ 
		 DLD &&& $ (-1,-3) $        & $ (-2-t_2,\, -4-t_1) $ \\ \bottomrule\\
	\end{tabular}
}
\caption{\label{ta:LSPends} Coordinates for the twelve unfolded positions of $ T=(t_1,t_2) $}
\end{table}
From the results in the table and the distance formula, we get, in Table~\ref{ta:LSPlen}, the squares of the lengths of the twelve candidates for shortest path.  Whichever of these candidates is shortest gives us the shortest  \LS-path or paths. (We have sometimes replaced a squared algebraic expression with its negative in the service of neater formatting.)  

In Table~\ref{ta:LSPlen},  we begin with a fixed source point~$ (s_1,s_2) $ in the base face of the unfolding grid and view the target point~$ (x,y)= (t_1,t_2) $ in the base face as varying around the interior of the superimposed target face.   The entries in Table~\ref{ta:LSPlen} solve the shortest path problem in principle, but in reality only up to computational precision in distinguishing between possibly very close lengths.  

\begin{table}[ht]
	%\scriptsize
\centering{
		\begin{tabular}{cl|cl} \toprule
			Roll Seq. & $ \text{Distance}^2 $ &  Roll Seq. & $ \text{Distance}^2 $ \\ \midrule
			RR  & $ (x+s_1-4)^2 + (y-s_2)^2 $ & LL    & $ (x+s_1+4)^2 + (y-s_2)^2 $   \\ 
			RUR  & $ (y+s_1-4)^2 + (x+s_2-2)^2 $  & LUL  & $ (y-s_1-4)^2 + (x-s_2+2)^2 $  \\  
			RDR  & $ (y-s_1+4)^2 + (x-s_2-2)^2 $  & LDL  & $ (y+s_1+4)^2 + (x+s_2+2)^2 $ \\ \midrule
			UU   & $ (x-s_1)^2 + (y+s_2-4)^2 $    & DD   & $ (x-s_1)^2 + (y+s_2+4)^2 $  \\
			URU  & $ (y+s_1-2)^2 + (x+s_2-4)^2 $  & DRD  & $ (y-s_1+2)^2 + (x -s_2-4)^2 $ \\ 
			ULU  & $ (y-s_1-2)^2 + (x-s_2+4)^2 $ & DLD  & $ (y+s_1+2)^2 + ( x+s_2+4)^2 $ \\ \bottomrule\\ 	
		\end{tabular}
	}
	\caption{\label{ta:LSPlen} Lengths squared of the 12 \LS- or pseudopaths with fixed source~$ (s_1,s_2) $ and variable target~$ (x,y) $}
\end{table}

\section{Endpoint distribution of 4-face shortest paths}
Like Mellinger-Viglione, our main interest is not to solve any one particular shortest path problem, but rather to understand the distribution of ``exceptional'' shortest paths.  More precisely, \emph{if we fix a source point, where are all target points, if any, with the property that the shortest path from source to target traverses four cube faces rather than three?} 

 The term ``4-face shortest path'' occurs frequently in what follows, so we abbreviate it to ``4FSP.''  Note that, by virtue of previous observations about the effect of corner moves, a 4FSP cannot be a pseudopath. If a particular path is shortest, then its endpoint belongs to the solution set of eleven inequalities among the twelve entries of Table~\ref{ta:LSPlen}. Since we want to plot the \emph{locations} of 4FSP target points rather than determine the 4FSP lengths, we can do with fewer inequalities.  If a given 4-face path~$ \mathcal{P} $, viewed on the cube surface, is shorter than all four 3-face paths to the endpoint of~$ \mathcal{P} $, then  either~$ \mathcal{P} $ is a 4FSP to its own endpoint, or some other 4-face path to the endpoint of~$ \mathcal{P} $ is even shorter. In either case, we have determined that the  endpoint of~$ \mathcal{P} $ is the endpoint of a 4FSP.  Thus, for each 4-face path in the unfolding grid, we only have to consider the four inequalities comparing it to the 3-face \LS-paths.

The inequalities involved are simpler than a cursory inspection of Table~\ref{ta:LSPlen} might suggest. Note that, when expanded, any single entry in the table contains the terms~$ x^2 $ and~$ y^2 $, with all other terms linear in~$ x $ and~$ y $ once values for~$ s_1 $ and~$ s_2 $ have been given.  The result is that whenever two entries are on opposite sides of an inequality, the~$ x^2 $ and~$ y^2 $ terms cancel. This leaves a linear inequality, defined by a straight line in the plane and consisting of all points in the plane on one side of that line.  This solution set, called a \emph{half-plane,} has to be intersected with the interior of the base face in the unfolding grid.  

So, pick a source point and one of the eight three-term roll sequences that correspond to 4-face paths  from the source point. Consider the set of all \LS\ or pseudo paths with the given source and roll sequence (the target point is the variable here). Among these paths may be some 4FSP's.  The inequality discussion above indicates that the set of endpoints of these paths is the intersection of four half-planes and the interior of the base face.  If this intersection is not empty, then it is necessarily the interior of a convex polygon. 

The leftmost diagram in Figure~\ref{f:4fp} shows how the intersectiom of solution sets of the four inequalities determines a region of 4FSP endpoints, in this case for the roll sequence~DRD.  The dotted lines identify the half-plane boundaries of the solution sets, and the arrows designate on which side of each line the solution set resides.  The intersection determining the region of 4FSP endpoints is outlined with a thicker line to make it easier to see.  The target point~$ T $ is located in the 4FSP endpoint region, and the corresponding 4FSP from~$ S $ to~$ T $ is labeled ``shortest'' in the figure.  

The six diagrams to the right in Figure~\ref{f:4fp} are numerical plots showing the 4FSP endpoint regions for various source point positions. The small lower leftmost diagram  illustrates the 4FSP endpoint regions for a source point in approximately the same position as in the main left diagram, with the region actually constructed in the main diagram outlined in the numerical plot.
 
\begin{figure}[thp]
	\centering
	\begin{minipage}{0.35\textwidth}
		\begin{overpic}[scale=0.11]{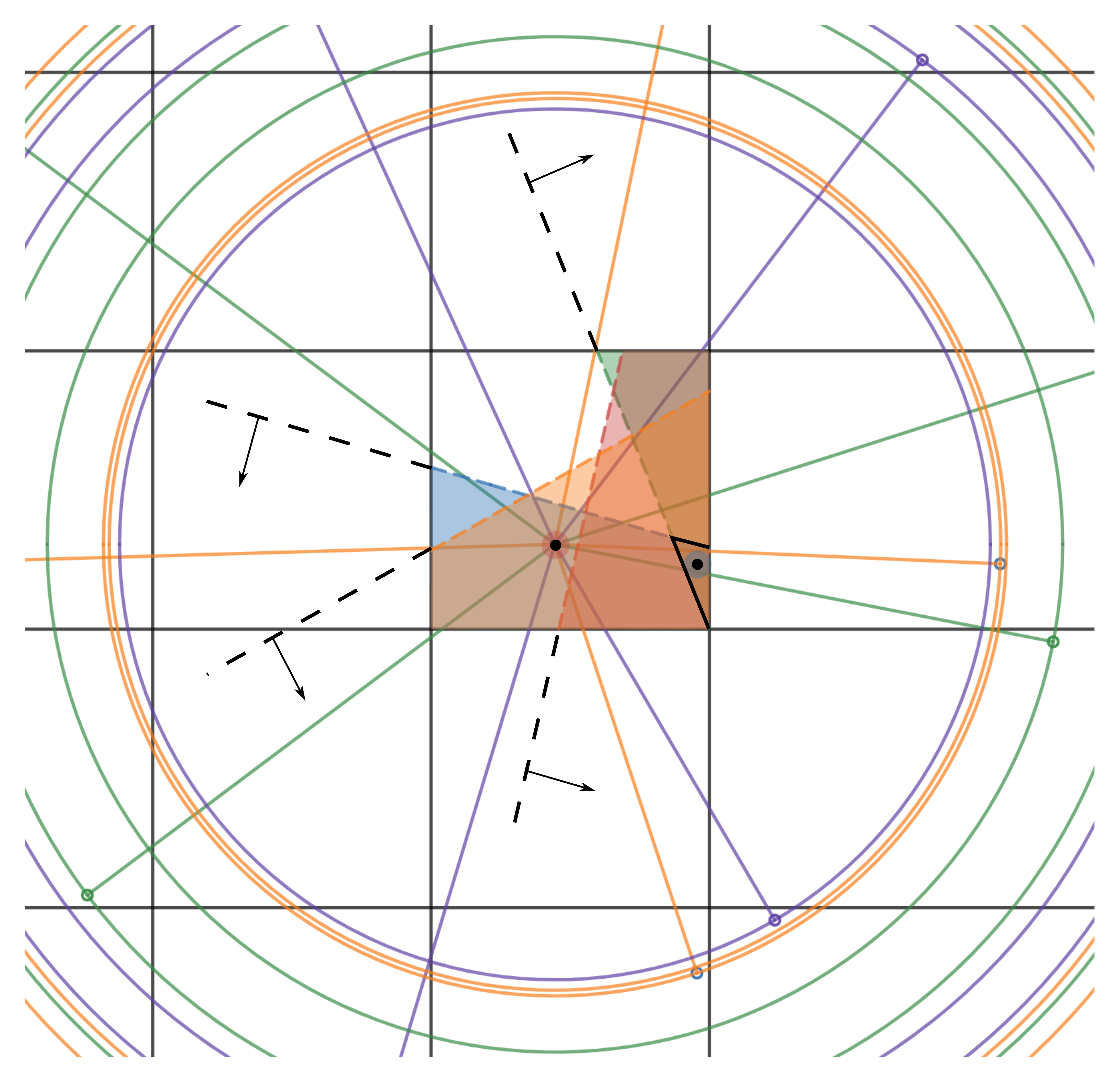} %\put(){\scriptsize }
			\put(46,43){\tiny $ S $}
			\put(64,43){\tiny $ T $}
			\put(50,75){\tiny $ DRD  <  DD $}
			\put(25,60){\tiny $ DRD  <  RR $}
			\put(9,43){\tiny $ DRD  <  UU $}
			\put(42,19){\tiny $ DRD  <  LL $}
			\put(62,28){\tiny Shortest}
			\put(95,24){$ \longrightarrow $}
		\end{overpic}
	\end{minipage}%
	\begin{minipage}{0.6\textwidth}
		\centering	
		\includegraphics[height=0.9in]{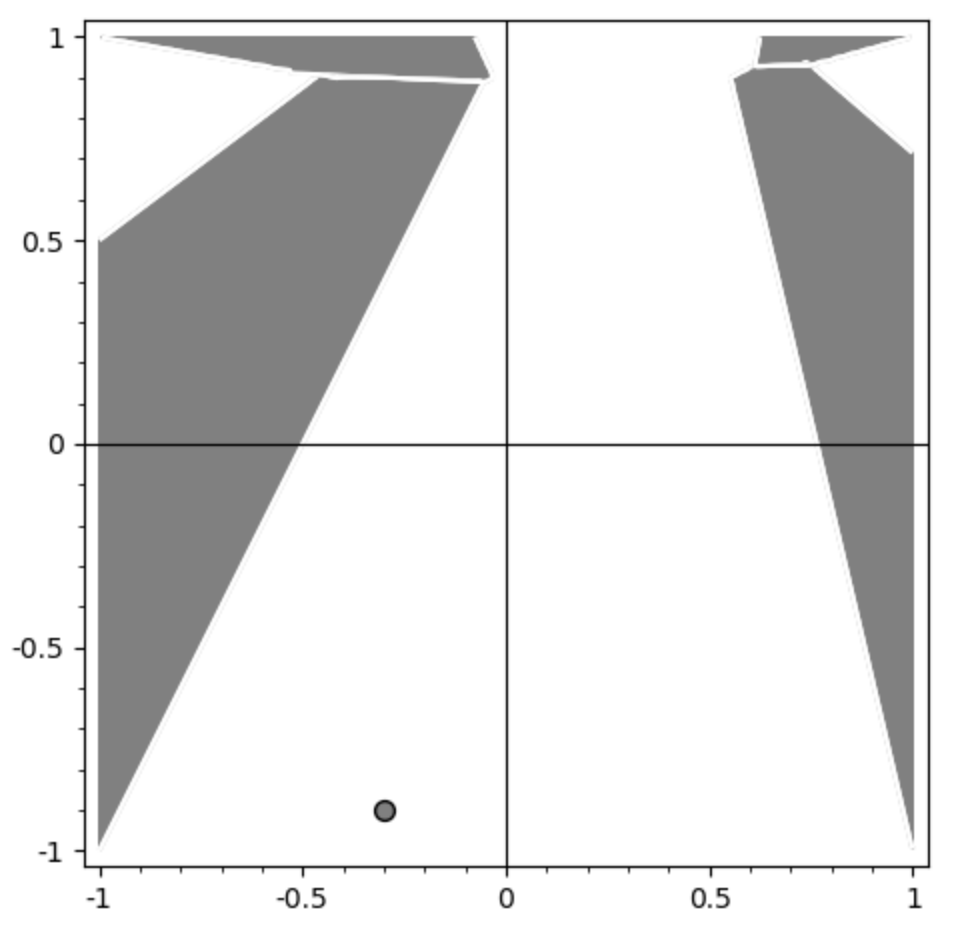} % Pt. at (-0.3,-0.9)
		\includegraphics[height=0.9in]{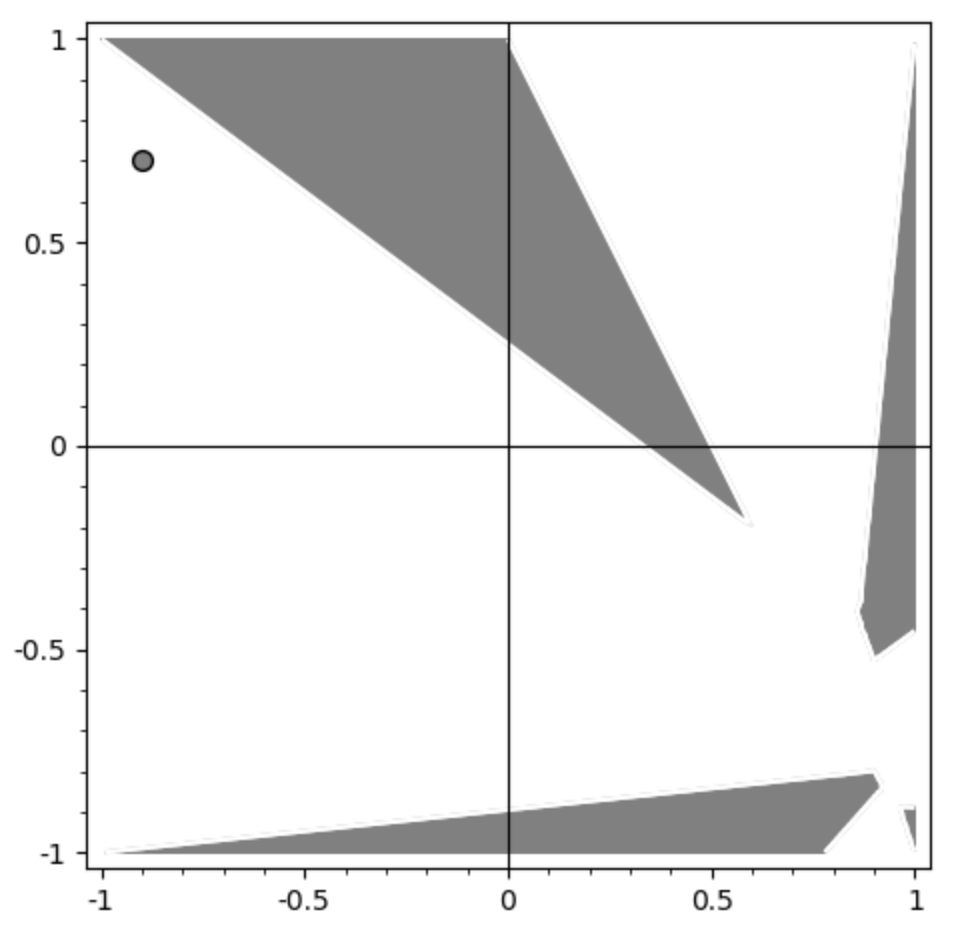} % Pt. at (-0.9,0.7)
		\includegraphics[height=0.9in]{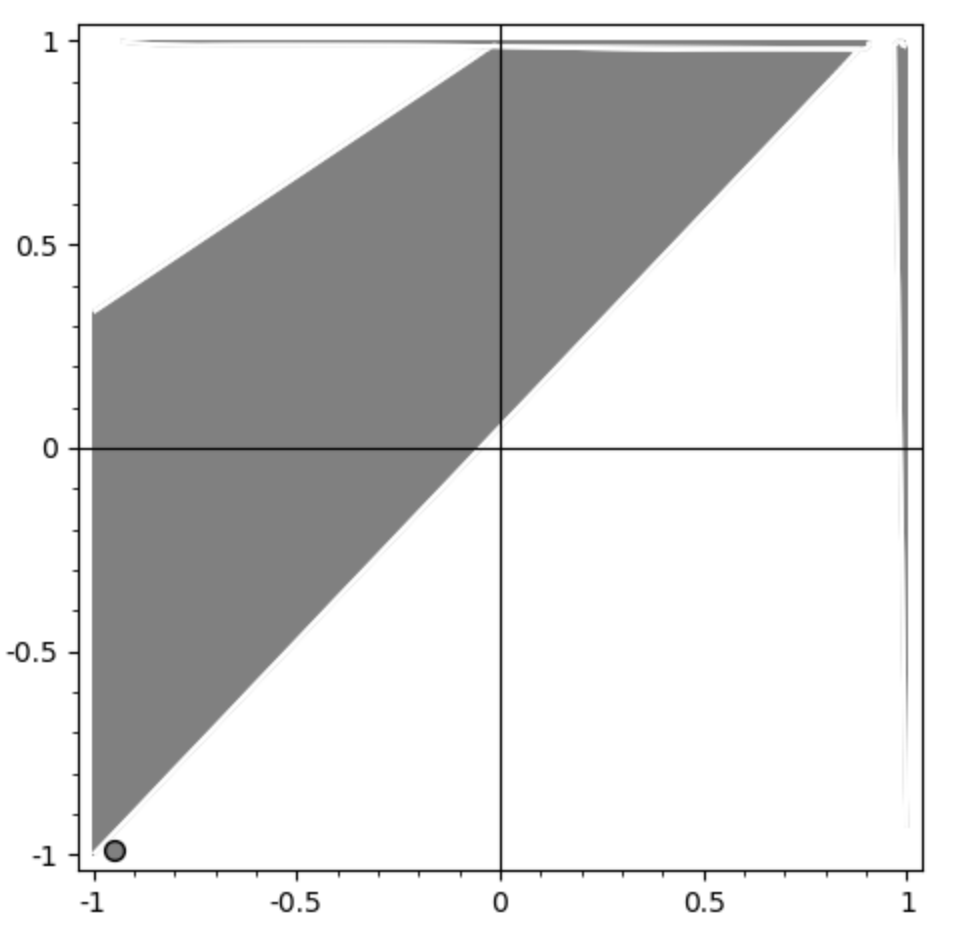} % Pt. at (-0.95, -0.99)
		\includegraphics[height=0.9in]{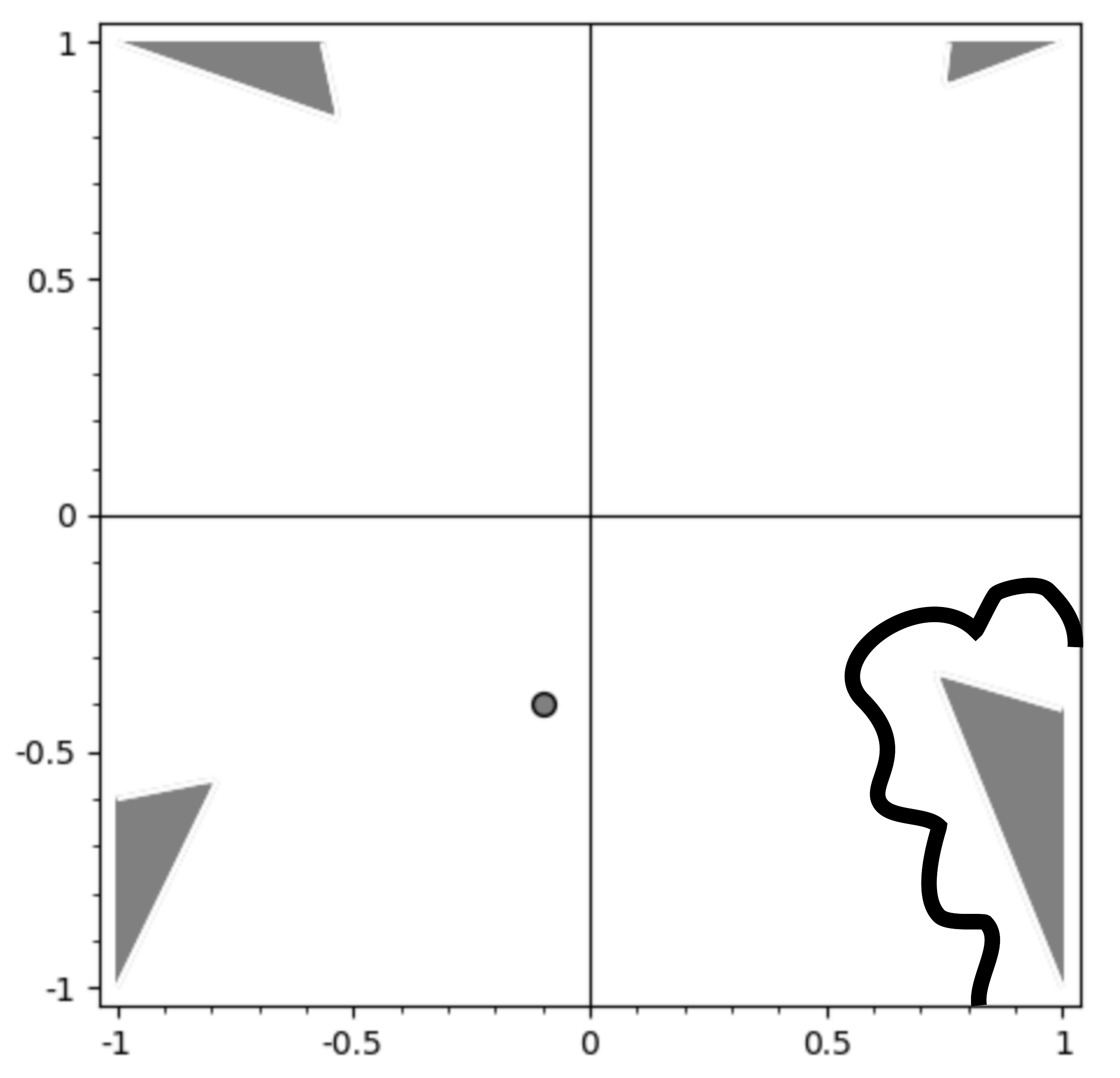} % Pt. at (-0.1,-0.4)
		\includegraphics[height=0.9in]{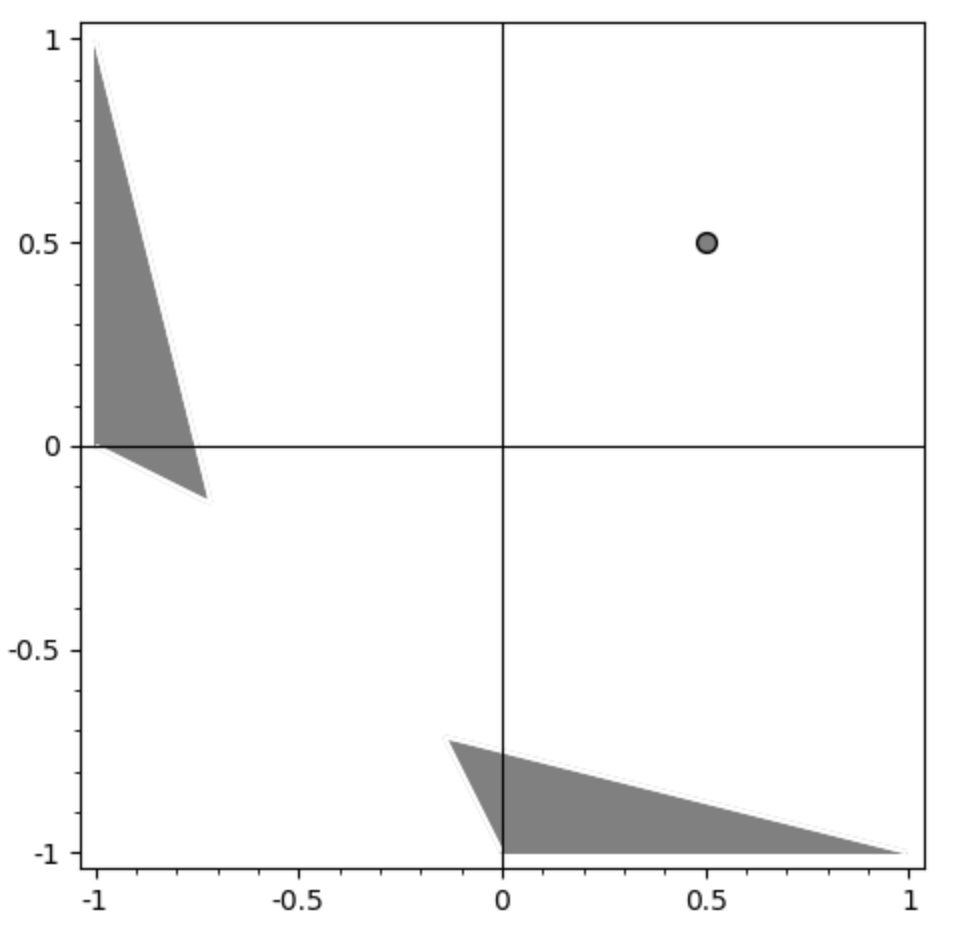} % Pt. at (0.5, 0.5)
		\includegraphics[height=0.9in]{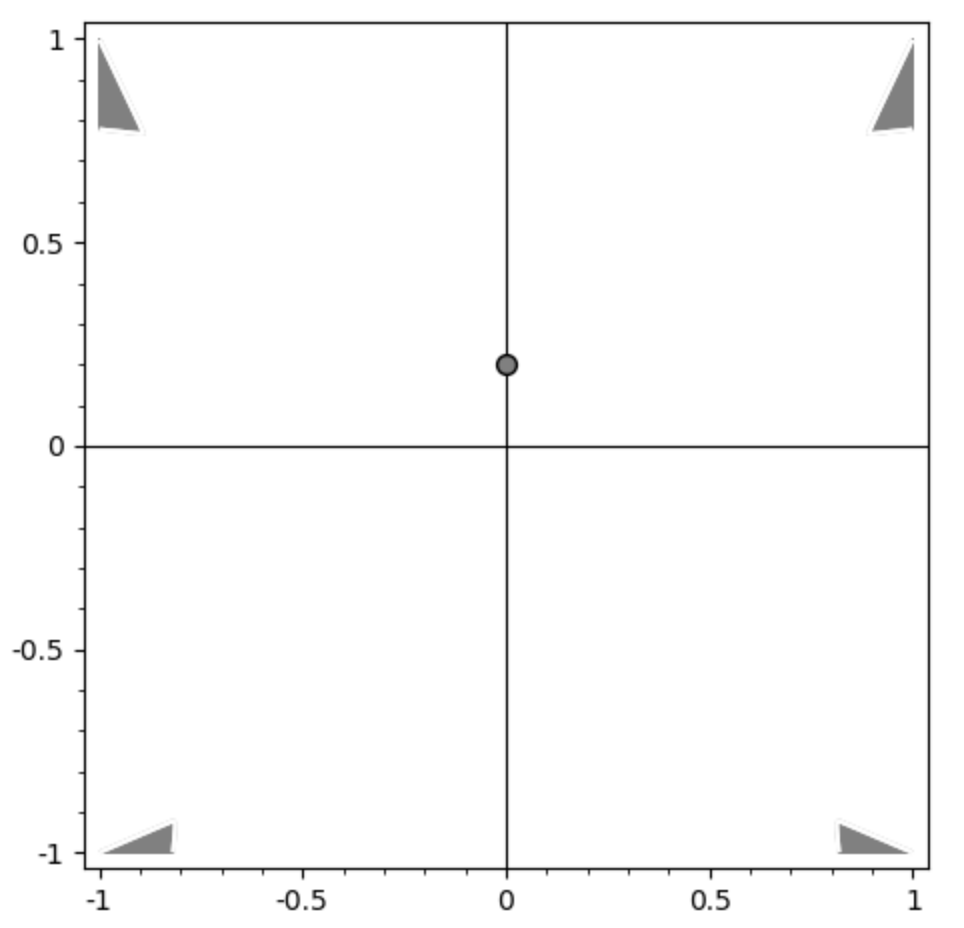} % Pt. at (0, 0.2)
	\end{minipage}
	\caption{\label{f:4fp} 4FSP endpoint regions}
\end{figure}

A region point could be plotted more than once, since there could be distinct equal-length shortest paths to a target point, with each path coming from a different roll sequence. The result is that some of the polygonal regions determined by this process might intersect, which is fine since it is the union of the regions we are interested in observing. The top left of the six  smaller diagrams in Figure~\ref{f:4fp} illustrates this possibility.

Two features of these numerical plots are notable. First, there appear to be at most four 4FSP endpoint regions, one at each corner. The eight roll sequences used to test for 4FSP target points could, in principle, produce as many as eight regions,  with as many as two at each corner. Second and just alluded to, the endpoint regions all seem to be ``anchored'' at a corner---no endpoint regions floating in mid-face, no endpoint regions adjacent to an edge but not to a vertex.  It turns out that both observations are correct.  Proofs are given in Appendix~B. 

Unions of convex polygons have well-defined areas, and this makes it possible to assign, to each possible source point, the probability that there will be a target point whose shortest path traverses four faces.  The probability is the area of the union of the  polygonal solution sets for each 4FSP type divided by the area of  the base face. In this way, we get a probability distribution on the source face of the cube, and our final result in this paper are plots of this distribution in Figure~\ref{f:4faceProp}.
\begin{figure}[h!]
	\centering
	\begin{overpic}[height=2.5in]{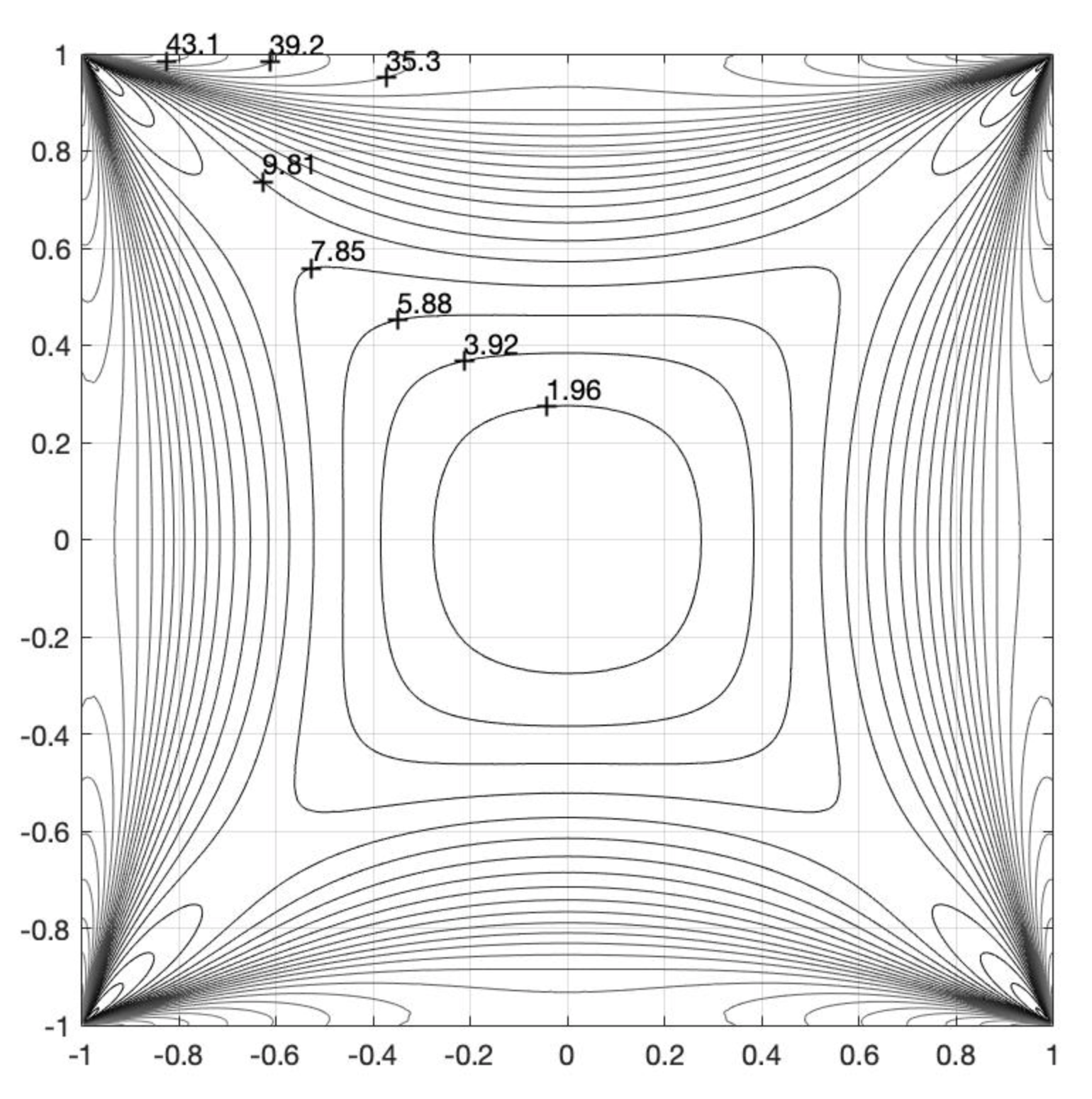}% \put(){\tiny }
		\put(49.7,49.5){\tiny $ + $}
		\put(52.2,52){\tinyzero}
	\end{overpic}
\qquad
    \begin{overpic}[height=2.5in]{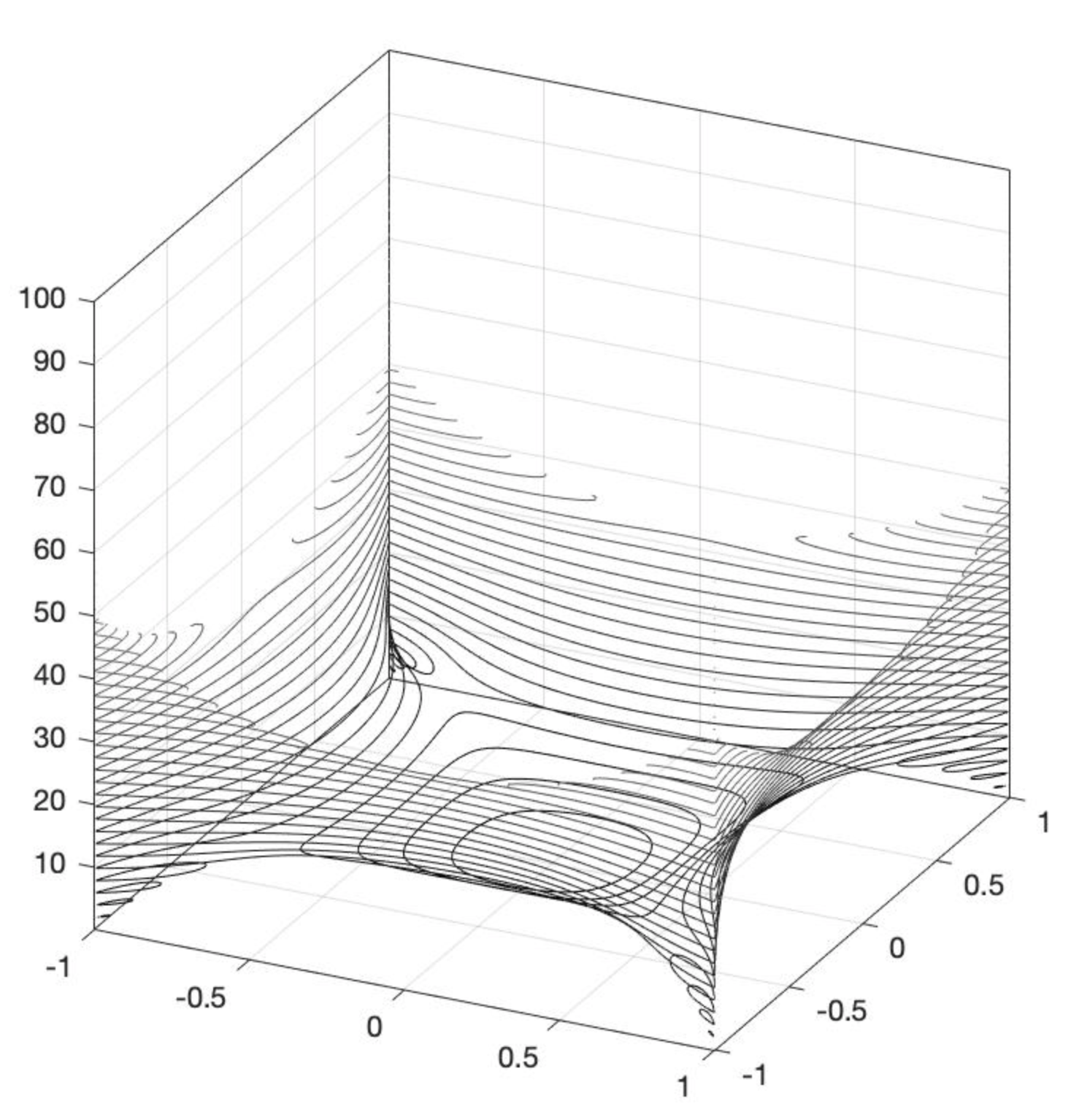}
    \end{overpic}
	\caption{Four face shortest path probability distribution (probabilites in percents) \label{f:4faceProp}}
\end{figure}

The polygonal area argument assures us that the probabilities are well-defined, but actually computing the area of a union of polygons for each each source point seems very complex.  Instead, we settle for estimates.  Our program searches a fine grid of points on the target face, identifies, via inequalities from Table~\ref{ta:LSPlen}, which points have 4FSP's, counts up the number of such points, and divides by the size of the full grid to get a probability estimate for each source point.  Further details of the program are found in Appendix~C.

The probability distribution shows that the source points most likely to have 4FSP's are located near the face vertices, in between face diagonals (where the probability is very low) and the face edges.  The face diagonals are low-probability areas because endpoint regions only occur at the non-diagonal vertices.  (For example, see the bottom middle diagram in Figure~\ref{f:4fp}.)  The probability is zero at the origin and near zero close to the origin.  The apparent singularities at the corners and the saddle points on the diagonals and on the edges are features whose explanation remains to be uncovered.

\appendix

\section{Appendix: Shortest Paths and Pseudopaths}

\subsection{Shortest paths have to be straight lines in the interior of some unfolding.}
 A path on the cube surface crosses a collection of cube edges, and these edges determine  a roll sequence that produces a partial unfolding in the unfolding grid. If the line segment~$ ST_1 $ joining~$ S $ to the unfolded image~$  T_1 $ of~$ T $ is a pseudopath for the roll sequence, then~$ ST_1 $ can be shortened to an \LS-path by a sequence of corner moves, which correspond to unfolding changes.  So there have to be unfoldings with straight line paths, contained in the interior of the unfolding, from~$ S $ to the unfolded image of~$ T $ (still called~$ T_1 $).  It is typically taken for granted that the shortest of these determines the shortest path on the cube surface from source to target.

To justify this supposition, suppose we have a shortest path on the cube surface from~$ S $ to~$ T $. It obviously has to be a straight line on any one of the cube faces, so, in the unfolding determined by the path, we get a piecewise linear (PL) path from~$ S $ to~$ T_1 $. (If the surface path passes through a vertex, there will be more than one possible roll sequence and partial unfolding.) 

 The straight line in the unfolding from~$ S $ to~$ T_1 $ is shorter than the PL path (unless the PL path turns out to be~$ ST_1 $).  Moreover, we claim that~$ ST_1 $ must be contained in the interior of the unfolding, i.e. must be an \LS-path.  If this is not so, i.e.~$ ST_1 $ is a pseudopath,  then a sequence of corner moves will produce an even shorter straight line, in the interior of its unfolding, from~$ S $ to an unfolded image~$ T_2 $ of~$ T $. But the path~$ ST_2 $ is an \LS-path shorter than~$ ST_1 $ which is shorter than the PL path, contradicting the minimality of the PL path's length. 

We now have the straight line in the unfolding from~$ S $ to~$ T_1 $  contained in the interior of the unfolding and so an \LS-path that would be shorter than the PL path, unless the PL path \emph{is} the straight line~$ ST_1 $, and this establishes the claim.

\subsection{Pseudopaths in the unfolding grid.}
Inequivalent roll sequences in the unfolding grid provide another interpretation of pseudopaths.   If we draw a straight line from a source point to a target point image in the unfolding grid, the squares the line passes through define a roll sequence. There are several roll sequences that land in the same target point square, corresponding to possible lines from source point to target point image.  For example, the sequences RUR, RRU, and URR all end at the same target square.  However, RUR is the only one of these sequences that corresponds to a path on the cube surface from the source face to the target face---which is \emph{opposite} the source face.  The other two roll sequences end on a face \emph{adjacent} to the source face, not opposite it.  The lines from source to target point image that produce these paths to adjacent faces do not lie in the interior of the RUR sequence of squares, and so have already been classified as pseudopaths. The point here is that these pseudopaths can also be interpreted as paths on the cube surface that fail to arrive at the correct target face.   

We could have chosen to disregard pseudopaths on the unfolding grid.  However, when computing path lengths, this would require that the program determine whether or not a path lies in the interior of an appropriate roll sequence.  Given that we are looking for shortest paths, and given that pseudopaths always determine a shorter  \LS-path via one or more corner moves, it is easier to simply compute all path lengths without trying to detect whether a path is pseudo or \LS, knowing that a pseudopath can never produce a ``false positive'' shortest result.

\section{Appendix: Qualitative Analysis of 4-face Regions}
Here is an overview of the situation analyzed in this appendix.  We have a source point on the source face of the cube and a target point on the target face.  The source and target faces are superimposed on the base face of the unfolding grid.  . This superposition of source and target faces leads to language that has to be properly understood.  For example, we have occasion to speak of source and/or target points lying on  a  base face diagonal.  This means that the source point is on a face diagonal of the cube's source face and the target point is on the corresponding diagonal of the cube's target face.  In the unfolding grid, these diagonals are superimposed and constitute a single base face diagonal.  

There are various paths between the source and target points on the cube, obtained from the unfolding grid by drawing straight lines from the source point to the twelve unfolded images of the target point. There are eight roll sequences  that produce 4-face \LS-paths or pseudopaths, and another four roll sequences that produce 3-face \LS-paths (see Table~\ref{ta:LSPlen}).  Our goal has been to determine, for each source point~$ S $, the regions on the cube's target face, if any, that are the endpoints of 4FSP's from~$ S $.

In the unfolding grid, these regions occur as the solution sets of systems of inequalities.  If we fix a source point~$ S =(s_1,s_2) $ and a roll sequence, then the target point coordinates~$ (x,y) $ are the only variables. (Later, we vary the source point coordinates~$ s_1 $ and~$ s_2 $ as well.)  The formulas for the length of the path associated with the roll sequence can be read off Table~\ref{ta:LSPlen}, and if that length is less than the lengths of the four 3-face \LS-paths (also in Table~\ref{ta:LSPlen}), then the target point~$ (x,y) $ is the endpoint of a 4FSP.

This means that for each source point~$ S $ and each roll sequence, there are four inequalities whose simultaneous solution set, intersected with the interior of the  unfolding grid base face, determines a region consisting of the endpoints of 4FSP's from~$ S $.

When we plot the 4FSP endpoint regions numerically, two features are notable. First, there appear to be at most four regions, one at each corner. The eight roll sequences could, in principle, produce as many as eight regions, two at each corner. Second and just alluded to, the endpoint regions are all ``anchored'' at a corner---no regions floating in mid-face, no regions adjacent to an edge but not to a vertex.   We shall see that this pair of observations is true in general.

In order to streamline the analysis, we avail ourselves of the isometries of the  cube source face, all of which extend to isometries of the entire cube. Using these isometries, we can restrict the source-point locations we have to consider.  The two median lines and the diagonals of the cube source face divide it into eight congruent right triangles, any two of which are related by either a rotation of the face, a reflection of the face, or the composition of both.  Since these transformations extend to the entire cube, they carry paths on the cube to other paths of the same length.  It follows that we can narrow our attention in the unfolding grid to source points~$ S=(s_1,s_2) $ in a single chosen triangle, getting the result for source points in any other triangle by applying an appropriate isometry to the chosen triangle. The triangle~$ \mathcal{T} $ we've chosen is the one whose vertices are~$ (0,0) $, $ (0,-1) $, and~$ (-1,-1) $ in the coordinate system used in Figure~\ref{f:unfgrid}.
 \begin{quote}
	\em All subsequent analyses assume that, in the unfolding grid base face, the source point~$ S= (s_1,s_2) $ is in triangle~$ \mathcal{T} $, including the two sides of~$ \mathcal{T} $ that are not a side of the base face, but excluding~$ (0,0) $.  This corresponds to the coordinate restrictions~$ -1 < s_2 \le s_1 \le 0 $ and $ (s_1,s_2) \ne (0,0) $.
\end{quote}
We have excluded the origin because we have already seen that there cannot be any 4FSP's with source point at the centroid of the source face.

\subsection{Why at most four regions when there are eight possiblities?}

The 4FSP candidates in the unfolding grid for a given source point and roll sequence have to be tested against the four 3-face \LS-paths from the same source point. If any one of these inequalities has a solution set that misses the interior of the  base face, then that inequality is not satisfied by the coordinates of any possible target point, in which case there is no need to check the other three inequalities.  As we shall argue next, this happens in four of the eight possibilities, and so leaves at most four non-empty regions of interest.  

Any one of the four inequalities just mentioned has a half-plane solution set.  This is for a fixed source point and roll sequence, but if we allow the source point to vary within the triangle $ \mathcal{T} $, then we find that \emph{the corresponding solution-set bounding lines all go through a single common point.} For example, consider paths corresponding to the roll sequence~ULU. One of the four inequalities a shortest path endpoint~$ (x,y) $ must satisfy requires the~ULU path to be shorter than the corresponding~LL path.  Consulting Table~\ref{ta:LSPlen}, we find the straight line~$ \mathcal{L}_{ (s_1,s_2)}(x,y) $ bounding the half-plane solution set for ``ULU path shorter than~LL path'' has equation 
\[ 
(y-s_1-2)^2 + (x-s_2+4)^2 = (x+s_1+4)^2 + (y-s_2)^2.
\]
Note that a potential target point~$ (x,y) $ will lie on~$ \mathcal{L}_{ (s_1,s_2)}(x,y) $ if we have
\[
y-s_1-2 = -(x+s_1+4) \quad\text{and}\quad x-s_2+4 = y-s_2.
\]
The appended minus sign in the first equality is legitimate, since the quantity is squared in the formula.  The sign has been chosen to make the occurrences of~$ s_1 $ cancel.  

The first equality simplifies to $ x+y=-2 $ and the second equality simplifies to $ x-y=-4 $. The simultaneous solution is $ (x,y)=(-3,1) $, a point common to all of the half-plane bounding lines, regardless of source point location, for ``ULU path shorter than~LL path.''

The fact that the lines~$ \mathcal{L}_{ (s_1,s_2)}(x,y) $ all pass through~$ (-3,1) $  allows us to see that the solution sets these lines bound have no points inside the  base face.  To establish this, we need the slopes of the bounding lines.  There are two ways to get the slope of~$ \mathcal{L}_{ (s_1,s_2)}(x,y) $. One way is to square everything, manipulate the result into the form $ y=mx+b $, and read off~$ m $, but a less tedious approach is to use implicit differentiation and solve for~$ dy/dx $, which gives
\[
 \frac{dy}{dx} = \frac{s_2+s_1}{s_2-s_1-2}.
\]
Recalling the restrictions~$ -1 < s_2 \le s_1 \le 0 $ and~$ (s_1,s_2) \ne (0,0) $, it is evident that for all source points in triangle~$ \mathcal{T} $, the line $ \mathcal{L}_{(s_1,s_2)}(x,y) $  has non-negative slope.  Since it also passes through~$ (-3,1) $,  it follows that~$ \mathcal{L}_{(s_1,s_2)}(x,y) $ is entirely above the target face. 

The half plane of solutions we are interested in comes from the associated inequality
\[
(y-s_1-2)^2 + (x-s_2+4)^2 < (x+s_1+4)^2 + (y-s_2)^2,
\] 
and it is easy to check that, say,~$ (x,y)= (-3,2) $ satisfies it, which means that the half-plane solution set lies above~$ \mathcal{L}_{(s_1,s_2)}(x,y) $.  Since~$ \mathcal{L}_{(s_1,s_2)}(x,y) $ is above the base face and the half-plane solution set is above~$\mathcal{L}_{(s_1,s_2)}(x,y) $, the solution set for the inequality within the base face is empty.  This means that regardless of the source-point location in the triangle~$ \mathcal{T} $, no 4-face path with roll sequence~ULU can ever be a shortest path.  

Analogous reasoning  shows that, for source points in the triangle~$ \mathcal{T} $, the roll sequences URU, RDR, and LDL can never have shortest 4-face paths, so that if there is a 4FSP, it would have to belong to one of the roll sequences RUR, DRD, DLD, and~LUL.  This explains why the plots show no more than four endpoint regions for 4FSP's, and why we never get two regions at any one vertex.

 There is a special case leading to additional exclusions.  When the source point lies on the hypotenuse of~$ \mathcal{T}$, then two of the above feasible roll sequences, RUR and DLD, are also excluded from ever having 4FSP's, as we shall see next.

The analysis for this situation is analogous to the ULU case considered above with one exception: in the ULU case, the half-plane of solutions to the corresponding inequality never intersects the interior of the base face.  In the two cases considered next, the half-plane of solutions can intersect the interior of the base face---but not when the source point is on the hypotenuse of~$ \mathcal{T}$.

Each of the roll sequences RUR and DLD has four inequalities to analyze; we illustrate what happens for~DLD. One of the four inequalities a shortest path endpoint~$ (x,y) $ must satisfy requires the~DLD path to be shorter than the~LL path.  Consulting Table~\ref{ta:LSPlen}, we find that the straight line~$ \mathcal{M}_{(s_1,s_2)}(x,y) $ bounding the half-plane solution set for~``DLD path shorter than~LL path'' has equation 
 \[
   (y+s_1+2)^2 + ( x+s_2+4)^2 = (x+s_1+4)^2 + (y-s_2)^2.
 \]
The same techniques used for the ULU case gives the common point for the lines~$ \mathcal{M}_{(s_1,s_2)}(x,y) $ to be~$ (-3,-1) $, and the slopes of the lines~$ \mathcal{M}_{(s_1,s_2)}(x,y) $ to be
\[
 \frac{dy}{dx} = \frac{s_1-s_2}{s_1+s_2+2} \qquad -1<s_2 \le s_1 \le 0, \quad (s_1,s_2) \ne (0,0).
 \]
For~$ s_1 $ and $ s_2 $ in the required range, we have $ dy/dx \ge 0 $, so the lines~$ \mathcal{M}_{(s_1,s_2)}(x,y) $ have non-negative slopes, with zero slopes if and only if~$ s_1 = s_2 $, in which case the line~$ \mathcal{M}_{(s_1,s_2)}(x,y) $ is the horizontal line through~$ (-3,-1) $ and so contains the bottom edge of the base face.

The associated inequality is
\[
(y+s_1+2)^2 + ( x+s_2+4)^2 < (x+s_1+4)^2 + (y-s_2)^2,
\]
and it is satisfied by~$ (-3,-2) $, indicating that the half-plane of solutions lies below the line~$ \mathcal{M}_{ (s_1,s_2)}(x,y) $.  In particular, for~$ \mathcal{M}_{(k,k)}(x,y) $, the half plane of solutions has the bottom edge of the base face as part of its boundary.   It follows that when the source point is on the hypotenuse of~$ \mathcal{T}$, there are no inequality solutions for ``DLD path shorter than LL path" in the interior of the target face, and so no DLD 4-face path can be shortest in this case. 
 
An analogous analysis excludes the roll sequence RUR when the source point is on the  hypotenuse of~$ \mathcal{T}$, so in that case there are only two feasibe roll sequences for 4FSP's, DRD and LUL.  This means there will be at most two regions of endpoints of 4FSP's in the base face.  In summary:

\begin{lem} \label{l:rollseq}
	Assume the source point~$ S $ is in the triangle~$ \mathcal{T} $.  If~$ S $ is not on the hypotenuse of~$ \mathcal{T}$, then the roll sequences capable of producing 4FSP's are RUR, DRD, DLD, and LUL.  If~$ S $ is on the hypotenuse of~$ \mathcal{T}$, then only DRD and LUL can produce 4FSP's.
 \end{lem}

\begin{cor}
	If the source point~$ S $ is on the hypotenuse of the triangle~$ \mathcal{T} $, then no target point anywhere on the base face diagonal can be the endpoint of a 4FSP.
\end{cor}
\begin{proof}
	In view of the lemma, we have to show that neither the DRD nor the LUL roll sequence can produce a 4FSP.  In fact, one of the 3-face \LS-paths obtained from RR or LL will be shorter.  We abuse our notation a bit and use $ |\text{DRD}| $, $ |\text{LUL}| $, $ |\text{RR}| $, and $ |\text{LL}| $ to denote the lengths of the paths obtained from the specified roll sequences.  With this notation, our task is to show that at least one of the two cases below always holds.
	
	Case 1: $ |\text{RR}|<|\text{DRD}| $. Consulting Table~\ref{ta:LSPlen} and setting $ x=y $ and $ s= s_1= s_2 $, we see that the path-length inequality to be established is 
	\[
	(x+s-4)^2 + (x-s)^2 < (x-s+2)^2 + (x-s-4)^2 ,
	\] 
	which after simplification becomes $ (s-1)(x-3)< 4 $. 
	
	Case 2: $ |\text{LL}|<|\text{LUL}| $. the path-length inequality to be established is
	\[
	(x+s+4)^2 + (x-s)^2 < (x-s+2)^2 + (x-s-4)^2,
	\]
	which after simplification becomes $ (s+3)(x+1)< 4 $.
	
	These inequalities are also subject to the restrictions~$ -1 < s < 0 $, that keeps~$ S $ on the hypotenuse of triangle~$ \mathcal{T} $, and~$ -1 < x < 1 $, that keeps~$ T $  in the interior of the base face.  The graphs in Figure~\ref{fig:stineq} show each inequality separately, the two inequalities simultaneously, and an enlarged portion of the simultaneous inequality plot for the restrictions just mentioned.  The enlarged view makes it clear that at least one of the two inequalities is always in force inside the restricted region, and this establishes the corollary.
    \end{proof} 

\begin{figure}
	\centering
	\begin{overpic}[width=\textwidth]{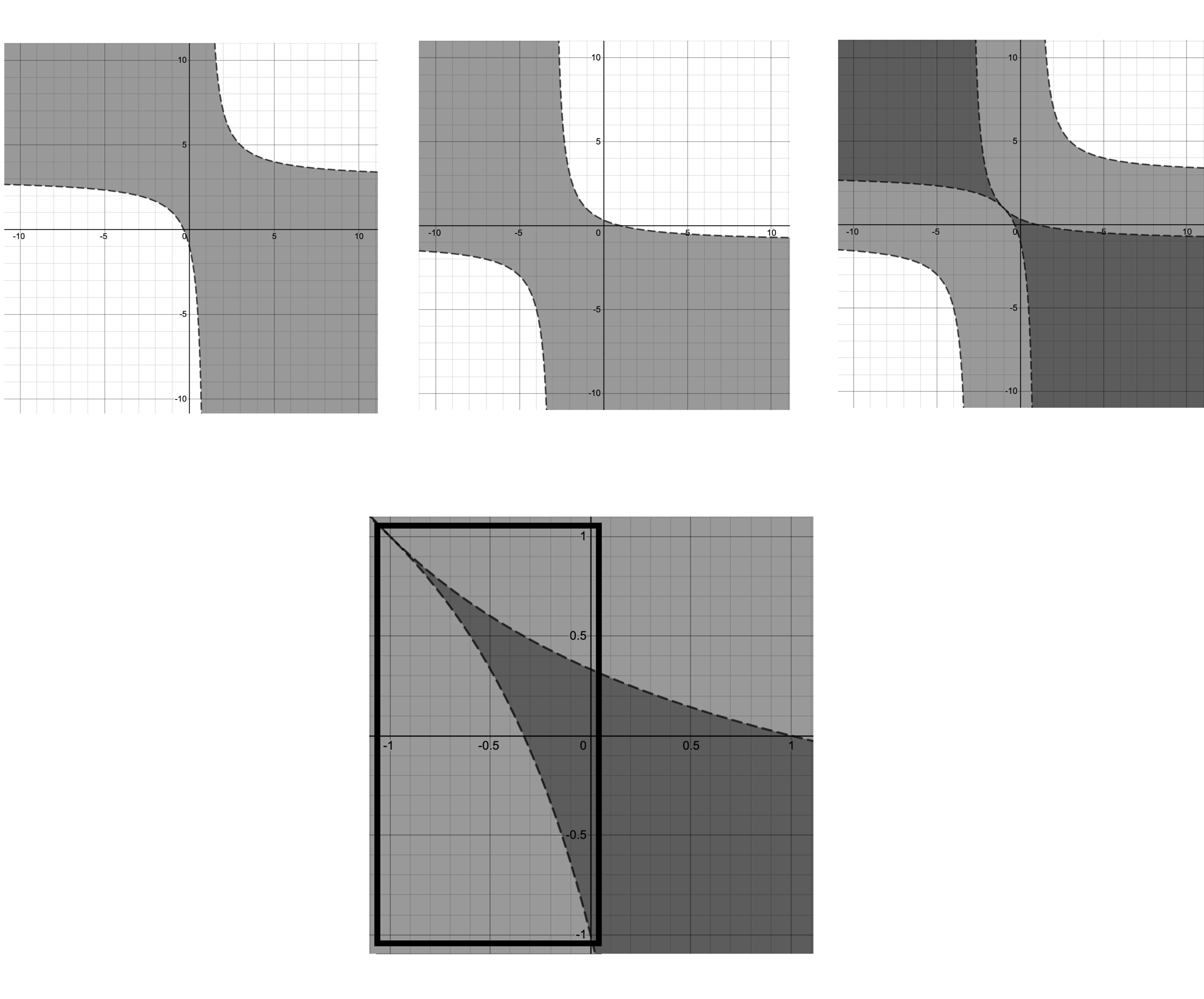} % \put(){\footnotesize}
		\put(16,77.5){\footnotesize $  x $}
		\put(29,64.5){\footnotesize $  s $}
		\put(5,46){\footnotesize $ (x-3)(s-1)<4 $ }
		\put(6,42.5){\footnotesize for $ |\text{RR}|<|\text{DRD}| $}
		\put(50.5,77.5){\footnotesize $  x $}
		\put(63,64.5){\footnotesize $  s $}
		\put(40,46){\footnotesize $ (x+1)(s+3)< 4 $ }
		\put(41.5,42.5){\footnotesize for $ |\text{LL}|<|\text{LUL}| $}
		\put(85,77.5){\footnotesize $  x $}
		\put(98,64.5){\footnotesize $  s $}
		\put(81,46){\footnotesize Together}
		\put(50.5,38){\footnotesize $  x $}
		\put(66,22){\footnotesize $  s $}
		\put(33,0.5){\footnotesize Framed: $ -1<s<0 $,\,\, $-1<x<1  $} 
	\end{overpic}
\caption{Inequalities for source and target on a base face diagonal \label{fig:stineq}}
\end{figure}
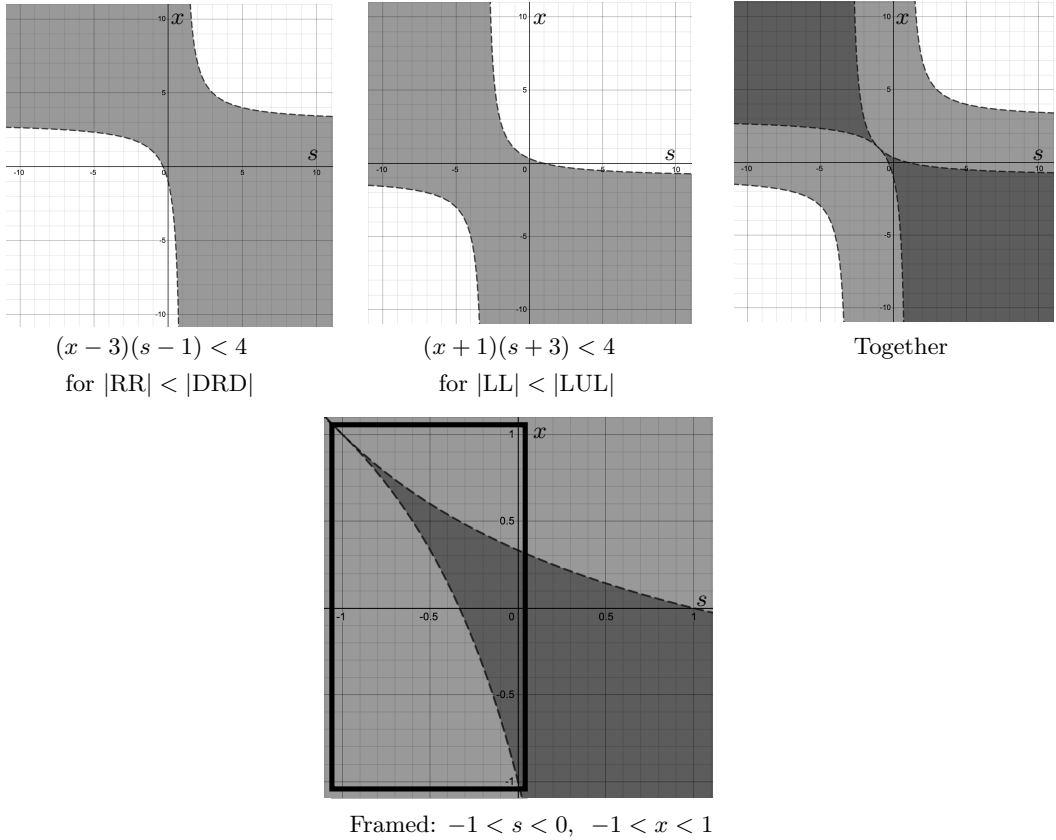

Applying the symmetries of the square, and referring back to the cube, the lemma says that, in general, if the source point is not on a source face diagonal, there are at most four regions of 4FSP endpoints on the target face, and if the source point is on a source face diagonal, there are at most two such regions on the target face, located at the corresponding off-diagonal corners.  We shall see below that these upper bounds are always attained.

\subsection{In what sense are the 4-face regions anchored to the corners?}
 ``Corner-anchoring'' is a consequence of the following proposition. 
 \begin{prop}
 	If, in the unfolding grid, the source point is not on a base face diagonal, then each corner of the base face has an open circular-sector neighborhood, all of whose points are endpoints of 4FSP's.  If the source point is  on a base face diagonal, then only  the off-diagonal corners have such neighborhoods.
 \end{prop} 
 
\begin{proof}
The assumption that the source point~$ S $ is in the triangle~$ \mathcal{T} $ is still in force.
\begin{figure}[h!]
	\centering
	\begin{overpic}[scale=0.3]{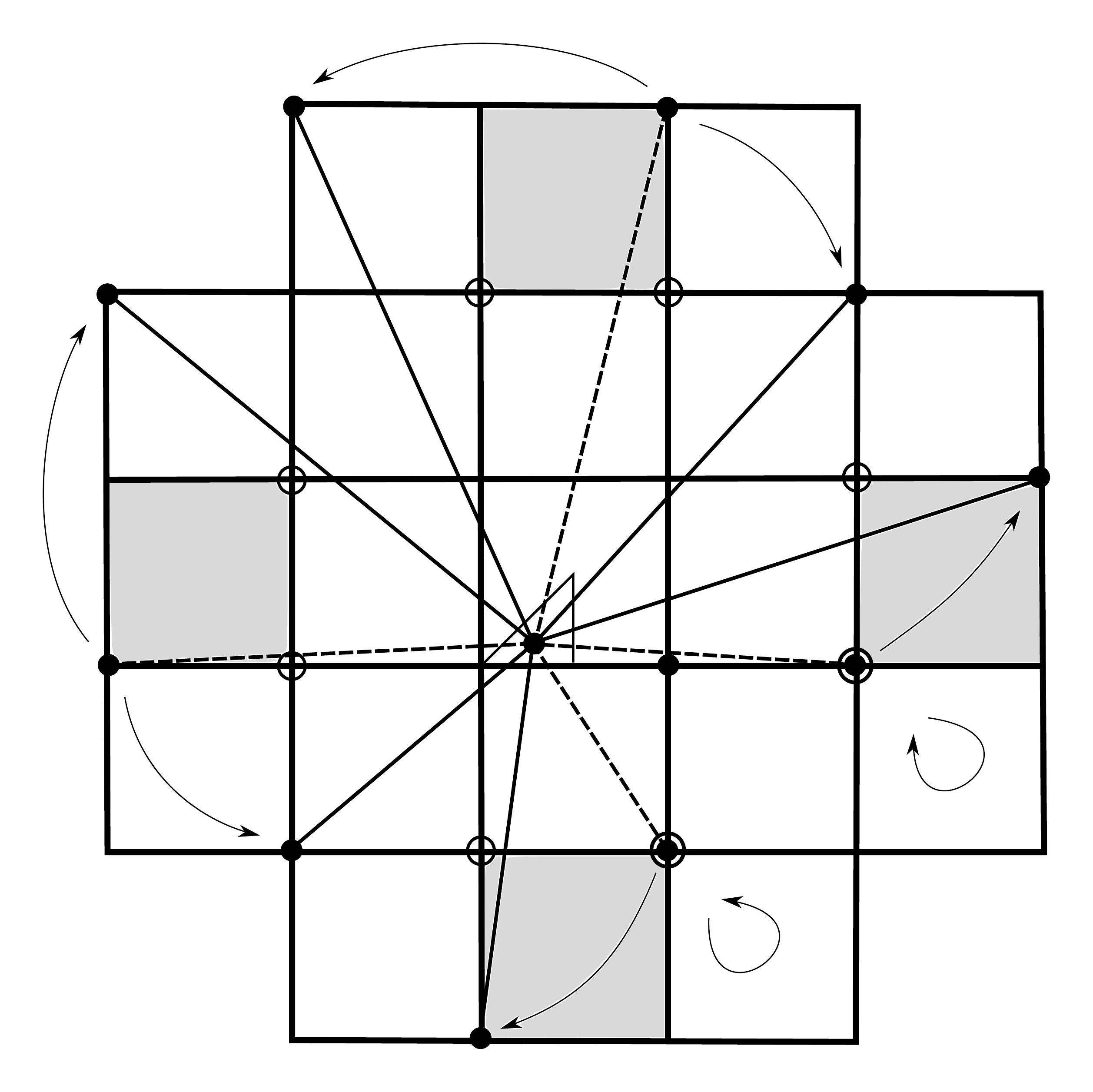} %\put(){\footnotesize }
		\put(47.2,45.5){\footnotesize $ S $}
		\put(62,35){\footnotesize $ T=C $}
		\put(34,33){\footnotesize $ d_1 $}
		\put(55,33){\footnotesize $ d $}
		\put(68,41){\footnotesize $ d_2 $}
		\put(62,18){\footnotesize $ P $}
		\put(80,35){\footnotesize $ Q_1 $}
		\put(98,55){\footnotesize $ Q_2 $}
		\put(80,75){\footnotesize $ Q_3 $}
		\put(62,92){\footnotesize $ Q_4 $}
		\put(21,92){\footnotesize $ Q_5 $}
		\put(6,76){\footnotesize $ Q_6 $}
		\put(3,37){\footnotesize $ Q_7 $}
		\put(21,18){\footnotesize $ Q_8 $}
		\put(42,0){\footnotesize $ Q_9 $}
	\end{overpic}
	\caption{\label{f:tarptcor} Unfolding grid with source point~$ S $ in the triangle~$ \mathcal{T} $ and target point~$ T $ at the corner~$C= (1,-1) $}
\end{figure}
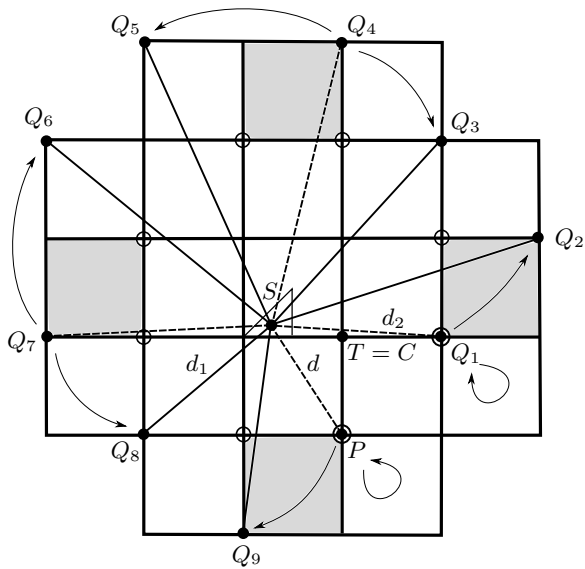  
We begin by finding the shortest paths in the unfolding grid from~$ S $ to each of the transformed images of the  base face vertices.  We then perturb~ $ T $ slightly, moving it off the vertex and into the interior of the  base face, where bona fide target points reside. The idea is that for small enough perturbations, the inequalities for~$ T $ that hold at the vertices will persist.

Assume first that~$ S $ is not on the hypotenuse of~$ \mathcal{T}$. There are four base face  vertices, with analogous situations at each one.  Here we give the details for the target point at the lower right corner vertex~$C=(1,-1)$ of the  base face.  Figure~\ref{f:tarptcor} illustrates the general unfolding grid situation for a source point~$ S $ in the triangle~$ \mathcal{T} $, the target point~$ T $ at the corner~$C$, and the ten target point images~$ P$ and~$Q_i $,~$ 1 \le i \le 9 $.  The dotted lines run from the source point to target-point images obtained by reflecting in the lines~$ x=\pm 2 $ and~$ y=\pm 2 $, corresponding to the roll sequences~RR, DD, LL, and~UU.  (If the target point was in the interior of the  base face, these dotted lines would be the 3-face \LS-paths.) The solid lines are obtained from the dotted line endpoints by corner moves.  The shaded faces are the ones that are rotated to realize corner moves, through either $ \pm 90^\circ $, around centers marked with open circles. The arrow arcs specify which solid path is associated with which dotted path via a corner move. The ``loops'' at~$ P $ and~$ Q_1 $ indicate that these points are fixed by the corner moves that use~$ P $ and~$ Q_1 $ as their respective centers of rotation.

In what follows, we use absolute value signs to denote path length.  We claim that for all positions of~$ S $ within the triangle~$ \mathcal{T} $, the path~$ SP $ is the unique shortest path to a cube vertex. Figure~\ref{f:tarptcor} indicates that we have to check nine inequalities~$ \abs{SP} < \abs{SQ_i} $,~$ 1 \le i \le 9 $,  to confirm this. Here we illustrate the inequality checks by comparing~$ SP $, with length~$ d $, to two of the nine possible alternatives.  One alternative,~$ SQ_8 $ with length~$ d_1 $, goes from the source point to~$ (-3,-3) $ in the unfolding grid, and another alternative, ~$ SQ_1 $ with length~$ d_2 $, goes from the source point to~$ (3,-1) $.  Taking~$ (x,y) $ for the source-point coordinates (note that previously~$ (x,y) $ has denoted a target point), we have
\begin{align*}
	d^2     &= (x-1)^2 + (y+3)^2, \\
	{d_1}^2 &= (x+3)^2 + (y+3)^2, \\
	{d_2}^2 &=  (x-3)^2 + (y+1)^2.	
\end{align*}
Basic algebra indicates that (1)~$ d < d_1 $ if and only if~$ x > -1 $, which is true for every source point in the interior of the base face, and (2)~$ d \le d_2  $ if and only if~$ y \le -x $, with $ d=d_2 $ if and only if~$ S=(0,0) $.  In Case~(2), the triangle~$ \mathcal{T} $, containing the source point is entirely below the line~$ y=-x $ except for a single point of contact at~$ (0,0) $, which we've excluded earlier.  For every other possible source point on the hypotenuse or in the interior of the triangle~$ \mathcal{T} $, the condition~$ y \le -x $ is satisfied and yields $ d < d_2 $. Analogous verifications with the other  seven inequalities indicate that the path~$ SP $ is the shortest \LS-path, so the closest target face vertex in the unfolding grid is at~$P=(1,-3) $ and no other vertex is as close.

Our strategy is now to budge the target point~$ T $ a small distance off the vertex into the interior of the  base face in a way that produces a 4-face \LS-path whose length is close enough to the original length~$ \lvert SP \rvert $ to be still shorter than all the other \LS-paths.  We plot the position of the target point relative to~$ C=(1,-1) $, so the displaced target point will be~$ T(\epsilon_h, \epsilon_v)=(1-\epsilon_h, -1+\epsilon_v) $ with~$ \epsilon_h,\, \epsilon_v > 0 $.  These choices mean that we now have $ T(0,0)=C $.  The displaced target point~$ T(\epsilon_h, \epsilon_v) $ has an an unfolding grid image~$ P(\epsilon_h, \epsilon_v) $ obtained by reflecting in the line $ y=-2 $ and rotating~$ 90^\circ $ counterclockwise around~$ P=P(0,0) $.

Of course, when~$ T(\epsilon_h, \epsilon_v) $, moves, all the other transformed images, originally the~$ Q_i $, move as well.  So we now denote the corresponding images by~$ Q_i(\epsilon_h, \epsilon_v) $, with the understanding that the original~$ Q_i $'s are now~$ Q_i(0,0) $.

If we want~$ T(\epsilon_h, \epsilon_v) $ to be the endpoint of a 4FSP, it is necessary (but not sufficient!) for the 4-face \LS-path~$ SP(\epsilon_h, \epsilon_v) $ to be shorter than its associated 3-face \LS-path. That 3-face \LS-path is not one of the~$ SQ_i $'s, since the associated pair only occurs because the motion of~$ T $ off of~$ C $ causes the path~$ SP $ to ``bifurcate'' into a 3-face and a 4-face \LS-path. In Figure~\ref{f:ffpaths}\,(a), the target point~$ T $ has been moved off~$ C $.  The dotted path~$ SP $ corresponding to the target point lying on~$ C $ now ``bifurcates'' into two paths, the 3-face \LS-path~$ SP' $ and the 4-face \LS-path~$ SP''$ obtained from~$ SP' $ by a corner move.  We want to at least keep the 4-face \LS-path shorter than the 3-face \LS-path.    This additional requirement imposes a~$ 135^\circ $ decision-angle condition on the image point~$P''= P(\epsilon_h, \epsilon_v) $; the constrained positions of~$ P(\epsilon_h, \epsilon_v) $  determine the lower shaded region in Figure~\ref{f:ffpaths}\,(b). That shaded region determines an equivalent upper shaded region of the base face.  From now on, we require that the variations in position governed by~$ \epsilon_h $ and~$ \epsilon_v $ keep~$ T(\epsilon_h, \epsilon_v) $---and so automatically~$ P(\epsilon_h, \epsilon_v) $---inside their respective shaded regions.
\begin{figure}[h!]
	\centering
	\begin{overpic}[scale=0.16]{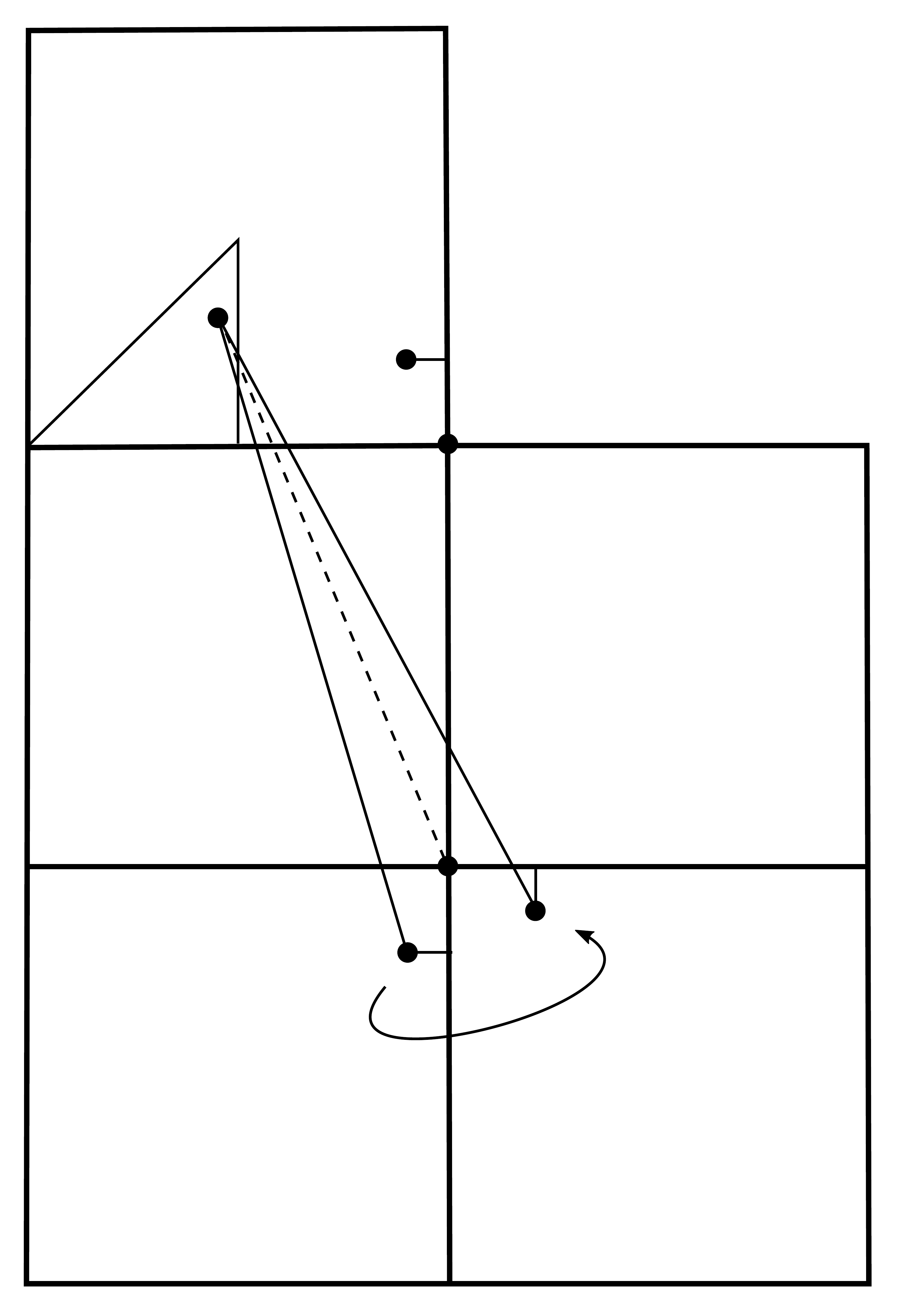} % \put(){\footnotesize }
		\put(7,67){\footnotesize $\mathcal{T}$ }
		\put(14,72){\footnotesize $ S $}
		\put(29,35){\footnotesize $ P $ }
		\put(27,72){\footnotesize $ T $}
		\put(36,61){\footnotesize $ C $}
		\put(29.5,23.5){\footnotesize $ P' $}
		\put(38,26){\footnotesize $ P'' $}
		\put(38,19){\scriptsize  corner move}
		\put(32,-2){\footnotesize (a)}
    \end{overpic}
\quad	
	\begin{overpic}[scale=0.16]{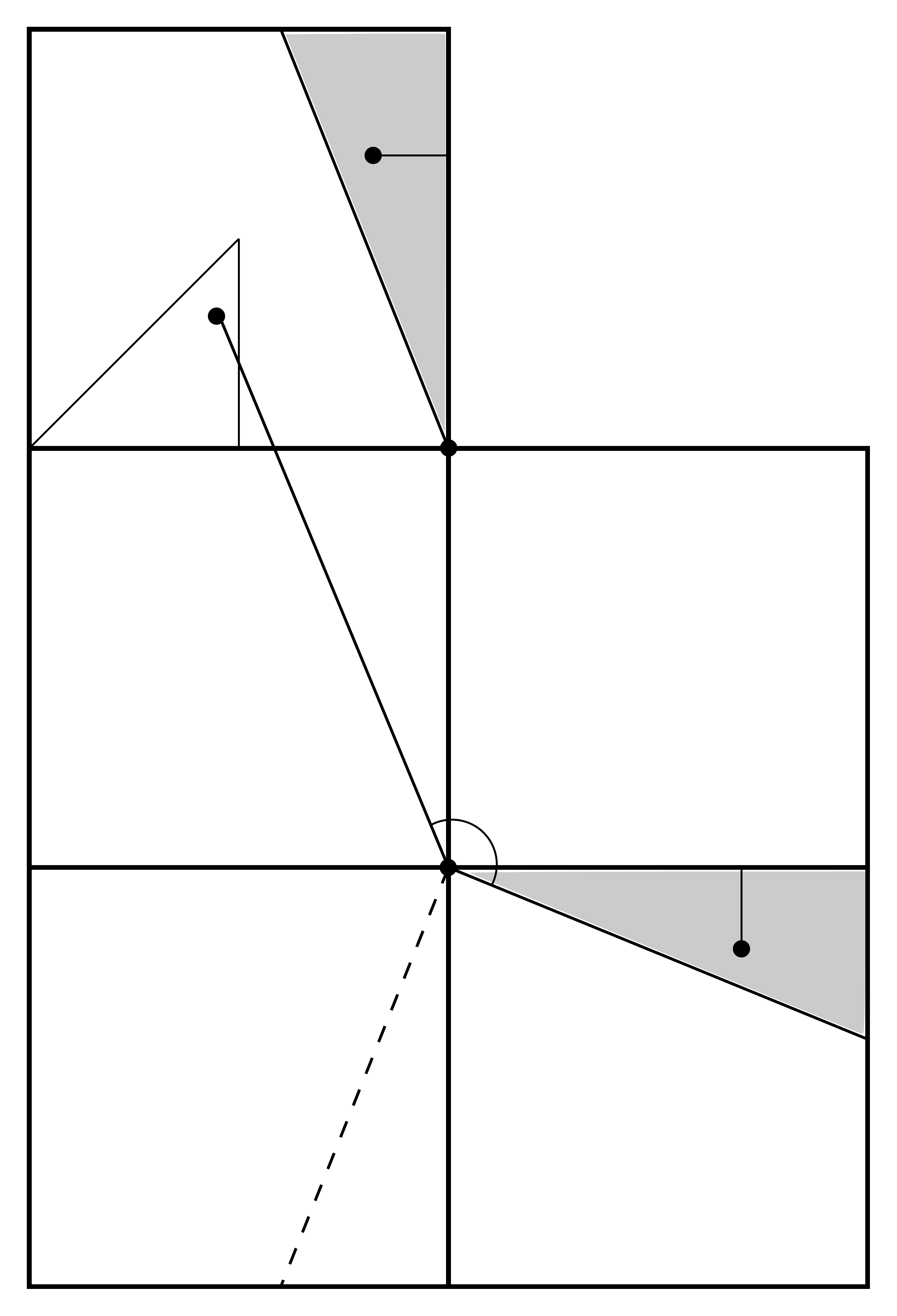} % \put(){\footnotesize }
		\put(7,67){\footnotesize $\mathcal{T}$ }
		\put(10,87){\footnotesize $ T(\epsilon_h,\epsilon_v) $}
		\put(14,72){\footnotesize $ S $}
		\put(30,90){\footnotesize $ \epsilon_h $}
		\put(35,79){\footnotesize $ \epsilon_v $}
		\put(36,61){\footnotesize $ C=(1,-1) $}
		\put(36,38){\footnotesize $ 135^\circ $}
		\put(29,30){\footnotesize $ P $ }
		\put(58,30){\footnotesize $ \epsilon_h $}
		\put(45,36){\footnotesize $ \epsilon_v $}
		\put(34.5,20){\footnotesize $ P''=P(\epsilon_h,\epsilon_v) $}
		\put(32,-2){\footnotesize (b)}
		
	\end{overpic}
\caption{Locating potential 4-face shortest paths \label{f:ffpaths}}
\end{figure}
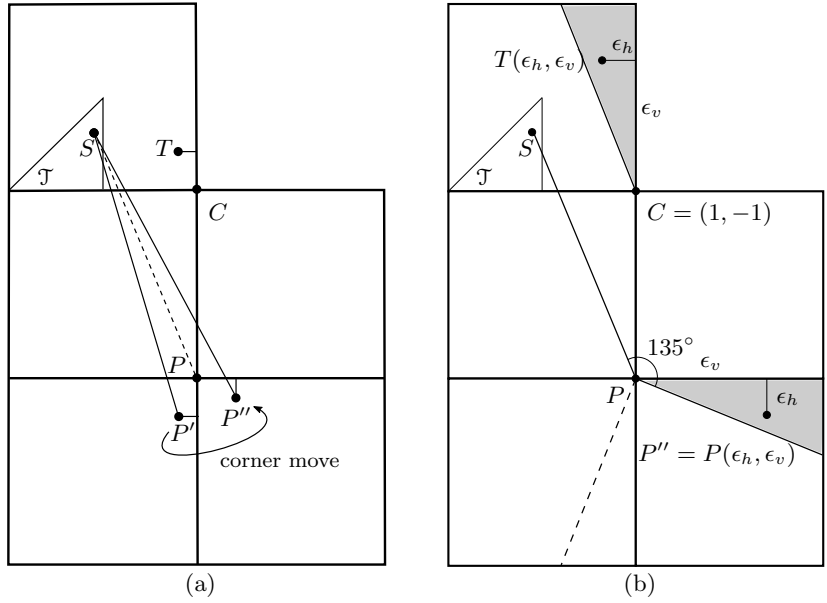

The fact that the path~$ SP $ is the shortest distance from~$ S $ to any of the unfolding grid images of~$ C $ is encoded in the inequalities
\begin{equation} \label{e:ineq1}
\abs{SQ_i(0,0)} - \abs{SP(0,0)} > 0, \quad 1 \le i \le 9, 
\end{equation} 
Now that the target point has been moved away from~$ C $, these inequalities have a new set of left-hand sides
$ \abs{SQ_i(\epsilon_h, \epsilon_v)} - \abs{SP(\epsilon_h, \epsilon_v)} $, together with the restriction of~$ P(\epsilon_h, \epsilon_v) $ to its shaded region. Each of these path-length differences is a continuous function~$ \varphi_i(\epsilon_h, \epsilon_v) $ that is positive at~$ (0,0) $, because at~$ (0,0) $ we get Equations~\eqref{e:ineq1}. Hence each~$ \varphi_i(\epsilon_h, \epsilon_v) $ must be positive in some neighborhood of~$ (0,0) $.  Taking the intersection of the open shaded decision-angle region and the nine neighborhoods of positivity, we obtain a single open circular-sector neighborhood~$ \mathcal{N} $.  Choosing~$ \epsilon_h $ and~$ \epsilon_v $ small enough to put~$ T(\epsilon_h, \epsilon_v) \in \mathcal{N} $, we have 
 \[
\abs{SQ_i(\epsilon_h, \epsilon_v)} - \abs{SP(\epsilon_h, \epsilon_v)} > 0, \quad 1 \le i \le 9,
\] 
and these updated inequalities, together with the decision angle region restriction, insure that~$ SP(\epsilon_h, \epsilon_v) $ determines a 4-face \LS-path that is the shortest path on the cube surface from~$ S $ to~$ T(\epsilon_h, \epsilon_v) $.  It follows that the entire open circular sector consists of endpoints of 4FSP's from~$ S $, and this verifies the proposition in the case the source and target are not both on a face diagonal.

The argument given above doesn't change if the source point is on the hypotenuse of triangle~$ \mathcal{T} $ and the target point is initially at~$ C $ as before. In view of the symmetries of the square, this shows that when the source point is on a base face diagonal (but not at the centroid) there will be 4FSP endpoint neighborhoods at the off-diagonal corners of the base face.

When the source point is on a base face diagonal, there is an argument, analogous to the one given earlier, that there are neighborhoods of the diagonal endpoints in which no  point is the endpoint of a 4FSP.  The point C from the previous argument is placed at the end of the diagonal containing~$ S $.  In this position, as we have seen in the corollary, the path inequalities favor a 3-face \LS-path, and the same type of continuity argument, in which the point~$ T $ is budged off of~$ C $, indicates that for small enough changes of position, the collection of path inequalities does not undergo any order changes. This means that when the source point is on a face diagonal, the endpoints of that diagonal  have  neighborhoods whose points cannot be the endpoints of 4FSP's.

\end{proof}

The existence of base face corner neighborhoods, all of whose points are endpoints of 4FSP's, shows that the upper bounds on the number of neighborhoods established earlier are always attained.  If the source point is not on a base face diagonal, then there are four neighborhoods, one at each corner of the base face. If the source point is on a base face diagonal, there are two neighborhoods, one on each off-diagonal corner of the base face.  Note, however, that these neighborhoods can be arbitrarily small, and so may not show up in a numerical plot.  Moreover, corner neighborhoods can be large enough to merge, creating a single neighborhood anchored to each of two target face vertices, as we see in the first plot of Figure~\ref{f:4fp}.

\section{Appendix: Program Listings}

\subsection{Program: endpoint face maps.}
The routine used to plot the 4FSP regions in Figure~\ref{f:4fp}  is at \url{https://tinyurl.com/y2rwnbyp}. The user has to enter coordinates for the source point~$ (S_1,S_2) $ on lines~5 and~6.  The program uses inequalities among the entries of Table~\ref{ta:LSPlen} to plot the regions in the target face whose points are the endpoints of 4FSP's from the specified source point. Here is a listing.
{\small
\begin{alltt}
	from sage.plot.scatter_plot import ScatterPlot
	var('x, y,  S_1, S_2' )
	
	# Source point coordinates
	S_1 = # Insert in number between -1 and 1
	S_2 = # Insert in number between -1 and 1
	
	# Path inequalities
	RDR1  = (x-S_2-2)^2 + (y-S_1+4)^2 < (x+S_1-4)^2 + (y-S_2)^2 
	RDR2  = (x-S_2-2)^2 + (y-S_1+4)^2 < (x+S_2-2)^2 + (y+S_1-4)^2  
	RDR3  = (x-S_2-2)^2 + (y-S_1+4)^2 < (x+S_2-4)^2 + (y+S_1-2)^2
	RDR4  = (x-S_2-2)^2 + (y-S_1+4)^2 < (x-S_2+4)^2 + (y-S_1-2)^2
	RDR5  = (x-S_2-2)^2 + (y-S_1+4)^2 < (x-S_1)^2   + (y+S_2-4)^2
	RDR6  = (x-S_2-2)^2 + (y-S_1+4)^2 < (x-S_2+2)^2 + (y-S_1-4)^2   
	RDR7  = (x-S_2-2)^2 + (y-S_1+4)^2 < (x+S_2+2)^2 + (y+S_1+4)^2
	RDR8  = (x-S_2-2)^2 + (y-S_1+4)^2 < (x+S_1+4)^2 + (y-S_2)^2 
	RDR9  = (x-S_2-2)^2 + (y-S_1+4)^2 < (x+S_2+4)^2 + (y+S_1+2)^2 
	RDR10 = (x-S_2-2)^2 + (y-S_1+4)^2 < (x-S_2-4)^2 + (y-S_1+2)^2 
	RDR11 = (x-S_2-2)^2 + (y-S_1+4)^2 < (x-S_1)^2   + (y+S_2+4)^2
	#------------------------------------------------------------
	RUR1 = (x+S_2-2)^2  + (y+S_1-4)^2 < (x+S_1-4)^2 + (y-S_2)^2 
	RUR2 = (x+S_2-2)^2  + (y+S_1-4)^2 < (x-S_2-2)^2 + (y-S_1+4)^2
	RUR3 = (x+S_2-2)^2  + (y+S_1-4)^2 < (x+S_2-4)^2 + (y+S_1-2)^2
	RUR4 = (x+S_2-2)^2  + (y+S_1-4)^2 < (x-S_2+4)^2 + (y-S_1-2)^2
	RUR5 = (x+S_2-2)^2  + (y+S_1-4)^2 < (x-S_1)^2   + (y+S_2-4)^2
	RUR6 = (x+S_2-2)^2  + (y+S_1-4)^2 < (x-S_2+2)^2 + (y-S_1-4)^2   
	RUR7 = (x+S_2-2)^2  + (y+S_1-4)^2 < (x+S_2+2)^2 + (y+S_1+4)^2
	RUR8 = (x+S_2-2)^2  + (y+S_1-4)^2 < (x+S_1+4)^2 + (y-S_2)^2 
	RUR9 = (x+S_2-2)^2  + (y+S_1-4)^2 < (x+S_2+4)^2 + (y+S_1+2)^2 
	RUR10 = (x+S_2-2)^2 + (y+S_1-4)^2 < (x-S_2-4)^2 + (y-S_1+2)^2 
	RUR11 = (x+S_2-2)^2 + (y+S_1-4)^2 < (x-S_1)^2   + (y+S_2+4)^2
	#------------------------------------------------------------
	URU1 = (x+S_2-4)^2  + (y+S_1-2)^2 < (x+S_1-4)^2 + (y-S_2)^2 
	URU2 = (x+S_2-4)^2  + (y+S_1-2)^2 < (x-S_2-2)^2 + (y-S_1+4)^2
	URU3 = (x+S_2-4)^2  + (y+S_1-2)^2 < (x+S_2-2)^2 + (y+S_1-4)^2
	URU4 = (x+S_2-4)^2  + (y+S_1-2)^2 < (x-S_2+4)^2 + (y-S_1-2)^2
	URU5 = (x+S_2-4)^2  + (y+S_1-2)^2 < (x-S_1)^2   + (y+S_2-4)^2
	URU6 = (x+S_2-4)^2  + (y+S_1-2)^2 < (x-S_2+2)^2 + (y-S_1-4)^2   
	URU7 = (x+S_2-4)^2  + (y+S_1-2)^2 < (x+S_2+2)^2 + (y+S_1+4)^2
	URU8 = (x+S_2-4)^2  + (y+S_1-2)^2 < (x+S_1+4)^2 + (y-S_2)^2 
	URU9 = (x+S_2-4)^2  + (y+S_1-2)^2 < (x+S_2+4)^2 + (y+S_1+2)^2 
	URU10 = (x+S_2-4)^2 + (y+S_1-2)^2 < (x-S_2-4)^2 + (y-S_1+2)^2 
	URU11 = (x+S_2-4)^2 + (y+S_1-2)^2 < (x-S_1)^2   + (y+S_2+4)^2
	#------------------------------------------------------------
	ULU1 = (x-S_2+4)^2  + (y-S_1-2)^2 < (x+S_1-4)^2 + (y-S_2)^2 
	ULU2 = (x-S_2+4)^2  + (y-S_1-2)^2 < (x-S_2-2)^2 + (y-S_1+4)^2
	ULU3 = (x-S_2+4)^2  + (y-S_1-2)^2 < (x+S_2-2)^2 + (y+S_1-4)^2
	ULU4 = (x-S_2+4)^2  + (y-S_1-2)^2 < (x+S_2-4)^2 + (y+S_1-2)^2
	ULU5 = (x-S_2+4)^2  + (y-S_1-2)^2 < (x-S_1)^2   + (y+S_2-4)^2
	ULU6 = (x-S_2+4)^2  + (y-S_1-2)^2 < (x-S_2+2)^2 + (y-S_1-4)^2   
	ULU7 = (x-S_2+4)^2  + (y-S_1-2)^2 < (x+S_2+2)^2 + (y+S_1+4)^2
	ULU8 = (x-S_2+4)^2  + (y-S_1-2)^2 < (x+S_1+4)^2 + (y-S_2)^2 
	ULU9 = (x-S_2+4)^2  + (y-S_1-2)^2 < (x+S_2+4)^2 + (y+S_1+2)^2 
	ULU10 = (x-S_2+4)^2 + (y-S_1-2)^2 < (x-S_2-4)^2 + (y-S_1+2)^2 
	ULU11 = (x-S_2+4)^2 + (y-S_1-2)^2 < (x-S_1)^2   + (y+S_2+4)^2
	#------------------------------------------------------------
	LUL1 = (x-S_2+2)^2  + (y-S_1-4)^2 < (x+S_1-4)^2 + (y-S_2)^2 
	LUL2 = (x-S_2+2)^2  + (y-S_1-4)^2 < (x-S_2-2)^2 + (y-S_1+4)^2
	LUL3 = (x-S_2+2)^2  + (y-S_1-4)^2 < (x+S_2-2)^2 + (y+S_1-4)^2
	LUL4 = (x-S_2+2)^2  + (y-S_1-4)^2 < (x+S_2-4)^2 + (y+S_1-2)^2
	LUL5 = (x-S_2+2)^2  + (y-S_1-4)^2 < (x-S_1)^2   + (y+S_2-4)^2
	LUL6 = (x-S_2+2)^2  + (y-S_1-4)^2 < (x-S_2+4)^2 + (y-S_1-2)^2 
	LUL7 = (x-S_2+2)^2  + (y-S_1-4)^2 < (x+S_2+2)^2 + (y+S_1+4)^2
	LUL8 = (x-S_2+2)^2  + (y-S_1-4)^2 < (x+S_1+4)^2 + (y-S_2)^2 
	LUL9 = (x-S_2+2)^2  + (y-S_1-4)^2 < (x+S_2+4)^2 + (y+S_1+2)^2 
	LUL10 = (x-S_2+2)^2 + (y-S_1-4)^2 < (x-S_2-4)^2 + (y-S_1+2)^2 
	LUL11 = (x-S_2+2)^2 + (y-S_1-4)^2 < (x-S_1)^2   + (y+S_2+4)^2
	#------------------------------------------------------------
	LDL1 = (x+S_2+2)^2  + (y+S_1+4)^2 < (x+S_1-4)^2 + (y-S_2)^2 
	LDL2 = (x+S_2+2)^2  + (y+S_1+4)^2 < (x-S_2-2)^2 + (y-S_1+4)^2
	LDL3 = (x+S_2+2)^2  + (y+S_1+4)^2 < (x+S_2-2)^2 + (y+S_1-4)^2
	LDL4 = (x+S_2+2)^2  + (y+S_1+4)^2 < (x+S_2-4)^2 + (y+S_1-2)^2
	LDL5 = (x+S_2+2)^2  + (y+S_1+4)^2 < (x-S_1)^2   + (y+S_2-4)^2
	LDL6 = (x+S_2+2)^2  + (y+S_1+4)^2 < (x-S_2+4)^2 + (y-S_1-2)^2 
	LDL7 = (x+S_2+2)^2  + (y+S_1+4)^2 < (x-S_2+2)^2 + (y-S_1-4)^2
	LDL8 = (x+S_2+2)^2  + (y+S_1+4)^2 < (x+S_1+4)^2 + (y-S_2)^2 
	LDL9 = (x+S_2+2)^2  + (y+S_1+4)^2 < (x+S_2+4)^2 + (y+S_1+2)^2 
	LDL10 = (x+S_2+2)^2 + (y+S_1+4)^2 < (x-S_2-4)^2 + (y-S_1+2)^2 
	LDL11 = (x+S_2+2)^2 + (y+S_1+4)^2 < (x-S_1)^2   + (y+S_2+4)^2
	#------------------------------------------------------------
	DLD1 = (x+S_2+4)^2  + (y+S_1+2)^2 < (x+S_1-4)^2 + (y-S_2)^2 
	DLD2 = (x+S_2+4)^2  + (y+S_1+2)^2 < (x-S_2-2)^2 + (y-S_1+4)^2
	DLD3 = (x+S_2+4)^2  + (y+S_1+2)^2 < (x+S_2-2)^2 + (y+S_1-4)^2
	DLD4 = (x+S_2+4)^2  + (y+S_1+2)^2 < (x+S_2-4)^2 + (y+S_1-2)^2
	DLD5 = (x+S_2+4)^2  + (y+S_1+2)^2 < (x-S_1)^2   + (y+S_2-4)^2
	DLD6 = (x+S_2+4)^2  + (y+S_1+2)^2 < (x-S_2+4)^2 + (y-S_1-2)^2 
	DLD7 = (x+S_2+4)^2  + (y+S_1+2)^2 < (x-S_2+2)^2 + (y-S_1-4)^2
	DLD8 = (x+S_2+4)^2  + (y+S_1+2)^2 < (x+S_1+4)^2 + (y-S_2)^2 
	DLD9 = (x+S_2+4)^2  + (y+S_1+2)^2 < (x+S_2+2)^2 + (y+S_1+4)^2
	DLD10 = (x+S_2+4)^2 + (y+S_1+2)^2 < (x-S_2-4)^2 + (y-S_1+2)^2 
	DLD11 = (x+S_2+4)^2 + (y+S_1+2)^2 < (x-S_1)^2   + (y+S_2+4)^2
	#------------------------------------------------------------
	DRD1 = (x-S_2-4)^2  + (y-S_1+2)^2 < (x+S_1-4)^2 + (y-S_2)^2 
	DRD2 = (x-S_2-4)^2  + (y-S_1+2)^2 < (x-S_2-2)^2 + (y-S_1+4)^2
	DRD3 = (x-S_2-4)^2  + (y-S_1+2)^2 < (x+S_2-2)^2 + (y+S_1-4)^2
	DRD4 = (x-S_2-4)^2  + (y-S_1+2)^2 < (x+S_2-4)^2 + (y+S_1-2)^2
	DRD5 = (x-S_2-4)^2  + (y-S_1+2)^2 < (x-S_1)^2   + (y+S_2-4)^2
	DRD6 = (x-S_2-4)^2  + (y-S_1+2)^2 < (x-S_2+4)^2 + (y-S_1-2)^2 
	DRD7 = (x-S_2-4)^2  + (y-S_1+2)^2 < (x-S_2+2)^2 + (y-S_1-4)^2
	DRD8 = (x-S_2-4)^2  + (y-S_1+2)^2 < (x+S_1+4)^2 + (y-S_2)^2 
	DRD9 = (x-S_2-4)^2  + (y-S_1+2)^2 < (x+S_2+2)^2 + (y+S_1+4)^2
	DRD10 = (x-S_2-4)^2 + (y-S_1+2)^2 < (x+S_2+4)^2 + (y+S_1+2)^2 
	DRD11 = (x-S_2-4)^2 + (y-S_1+2)^2 < (x-S_1)^2   + (y+S_2+4)^2

	# Region plots
	DRD = region_plot([DRD1, DRD2, DRD3, DRD4, DRD5, DRD6, DRD7, 
	                   DRD8, DRD9, DRD10, DRD11], (x,-0.999,0.999), 
	                   (y,-0.999,0.999), 
	                   plot_points=200, incol='grey' )
	DLD = region_plot([DLD1, DLD2, DLD3, DLD4, DLD5, DLD6, DLD7,
	                   DLD8, DLD9, DLD10, DLD11], (x,-0.999,0.999), 
	                   (y,-0.999,0.999), 
	                   plot_points=200, incol='grey' )
	LDL = region_plot([LDL1, LDL2, LDL3, LDL4, LDL5, LDL6, LDL7,
	                   LDL8, LDL9, LDL10, LDL11], (x,-0.999,0.999), 
	                   (y,-0.999,0.999),
	                   plot_points=200, incol='grey' )
	LUL = region_plot([LUL1, LUL2, LUL3, LUL4, LUL5, LUL6, LUL7,
	                   LUL8, LUL9, LUL10, LUL11], (x,-0.999,0.999), 
	                   (y,-0.999,0.999),
	                   plot_points=200, incol='grey' )
	ULU = region_plot([ULU1, ULU2, ULU3, ULU4, ULU5, ULU6, ULU7,
	                   ULU8, ULU9, ULU10, ULU11], (x,-0.999,0.999), 
	                   (y,-0.999,0.999),
	                   plot_points=200, incol='grey' )
	URU = region_plot([URU1, URU2, URU3, URU4, URU5, URU6, URU7,
	                   URU8, URU9, URU10, URU11], (x,-0.999,0.999), 
	                   (y,-0.999,0.999),
	                   plot_points=200, incol='grey' )
	RUR = region_plot([RUR1, RUR2, RUR3, RUR4, RUR5, RUR6, RUR7,
	                   RUR8, RUR9, RUR10, RUR11], (x,-0.999,0.999), 
	                   (y,-0.999,0.999),
	                   plot_points=200, incol='grey' )
	RDR = region_plot([RDR1, RDR2, RDR3, RDR4, RDR5, RDR6, RDR7,
	                   RDR8, RDR9, RDR10, RDR11], (x,-0.999,0.999), 
	                   (y,-0.999,0.999),
	                   plot_points=200, incol='grey' )

	# Source point    
	P = scatter_plot([[S_1,S_2]], facecolor = 'grey')
	
	# Combined plot of all regions
	show(P+DRD+DLD+LDL+LUL+ULU+URU+RUR+RDR, frame=True, aspect_ratio=1)
\end{alltt} 
} 
\subsection*{Program: source point probability distribution.} 
Unions of convex polygons have well-defined areas, and this makes it possible to assign to each possible source point the probability that there will be a target point whose shortest path traverses four faces. This probability is found by dividing the union of all of the 4FSP polygonal regions by the area of the face of the cube. By computing this probability for each source point, we find a probability distribution that can be displayed in the form of a contour plot (seen in Figure~\ref{f:4faceProp}).

In computing these areas and probabilities we have settled for estimations, as  exact computations seem very difficult. We use a program written in MATLAB. The program utilizes source and target points that are uniformly distributed 0.01 units away from each other on the source and target faces of the cube. For each target point, the program calculates the shortest path(s) from the source point, identifies whether any of those paths crosses four faces, and keeps count of the ones that do. After repeating this process and so counting 4FSP's from the given source point to all target points, the resulting count is divided by the total number of target points tested; this is the probability estimate for a given source point.

The program applies this process to all of the source points on the cube, and outputs three column vectors of data: the $ x $-component of the source point, the $ y $-component of the source point, and the probability for the given source point at~$ (x,y) $. This data is then reconfigured into matrices to be readable by MATLAB’s contour function, which outputs Figure~\ref{f:4faceProp}. Here is a listing.

\medskip\noindent\textbf{Cube Code for a single source point}
\begin{alltt}\small
	S1 = input('Enter the x coordinate of the source point ');
	S2 = input('Enter the y coordinate of the source point ');
	while S1 < -1 || S1>1
	S1 = input('Please enter a valid x coordinate for the source point ');
	end
	while S2 < -1 || S2>1
	S2 = input('Please enter a valid y coordinate for the source point ');
	end
	for i= -1:.02:1  %i is the "row":    the y-coord of the targ pts
	for j= -1:.02:1  %j is the "column": the x-coord of the targ pts
	   %------------------------------------------------------------------
	   % Note that the program runs through the target points that 
	   % were projected onto the source square and will ultimately 
	   % be plotted back on this square. However, we must first 
	   % translate to the 12 locations first to determine the 
	   % shortest paths.
	   %------------------------------------------------------------------
	[HR1, HR2, HL1, HL2, VU1, VU2, VD1, VD2, HR41A, HR42A, HR41B,
	 HR42B, HL41A, HL42A, HL41B, HL42B, VU41A, VU42A, VU41B, VU42B, 
	 VD41A, VD42A, VD41B, VD42B]=TargetCoords(j,i); 
	   %------------------------------------------------------------------
	   % Function takes in the coordinate of the target point and outputs 
	   % all of the 12 translations.
	   %------------------------------------------------------------------
	[minDistVal, minIndex] = MinDistance(S1, S2, HR1, HR2, HL1, HL2, VU1, 
	VU2, VD1, VD2, HR41A, HR42A, HR41B, HR42B, HL41A, HL42A, HL41B, HL42B, 
	VU41A, VU42A, VU41B, VU42B, VD41A, VD42A, VD41B, VD42B);
	   %------------------------------------------------------------------ 
	   % Finds the minimum distance from the 12 translations to the source 
	   % point, gives the index.
	   %------------------------------------------------------------------ 
	Plotting(minIndex, j, i) % Plots point, determines which color to assign
	end
	end
	plot(S1,S2,'s','MarkerFaceColor','cyan','MarkerSize', 10); 
	   % Plots the source point for context.
	
\end{alltt}

\medskip\noindent\textbf{Target Coordinates Function}
{\small
\begin{alltt}
	function [HR1, HR2, HL1, HL2, VU1, VU2, VD1, VD2, HR41A, HR42A, HR41B, 
	          HR42B, HL41A, HL42A, HL41B, HL42B, VU41A, VU42A, VU41B, VU42B, 
	          VD41A, VD42A, VD41B, VD42B] = TargetCoords(T1, T2)
	   %------------------------------------------------------------------
	   % TargetCoords finds the target point coordinates for each roll 
	   % Horizontal left and right and Vertical up and down.
	   % Following previous notations, 1 is for x-values, 2 is for y-values
	   % HR: Horizontal Right, HL: Horizontal Left, VU: Vertical Up, VD: Vertical
	   % Down.
	   %------------------------------------------------------------------
	T1 = round(T1,3);
	T2 = round(T2,3);
       % 3-Face Paths:
	HR1 = 4 - T1;
	HR2 = T2;
	HL1 = -4 -T1;
	HL2 = T2;
	VU1 = T1;
	VU2 = 4-T2;
	VD1 = T1;
	VD2 = -4-T2;
	   %------------------------------------------------------------------
	   % 4-Face Paths:
	   % Notation: The HR, HL, VU, VD are the same from above, as are the 1s and
	   % 2s. The 4 represents the 4-face path and the A,B is for each 4-face
	   % path associated with the general "rolling out".
	   %------------------------------------------------------------------
	HR41A = 4 - T2;
	HR42A = 2 - T1;
	HR41B = T2 + 4;
	HR42B = T1 - 2;
	HL41A = T2 - 4;
	HL42A = T1 + 2;
	HL41B = -T2 - 4;
	HL42B = -2 - T1;
	VU41A = 2 - T2;
	VU42A = 4 - T1;
	VU41B = T2 - 2;
	VU42B = T1 + 4;
	VD41A = 2 + T2;
	VD42A = T1 - 4;
	VD41B = -2 - T2;
	VD42B = -4 -T1;
	end
\end{alltt}
}
\medskip\noindent\textbf{Minimum Distance Function}
{\small
\begin{alltt}
function [minDistVal, minIndex] = MinDistance(S1, S2, HR1, HR2, HL1, HL2, VU1, VU2,
	                                              VD1, VD2, HR41A, HR42A, HR41B, HR42B,
	                                              HL41A, HL42A, HL41B, HL42B, VU41A,
	                                              VU42A, VU41B, VU42B, VD41A, VD42A,
	                                              VD41B, VD42B)
	   %------------------------------------------------------------------
	   % MinDistance finds all distances between the rolled out points and source
	   % points. These distances are stored in the array Dist and then the minimum
	   % distance and the index it occured at will be found. 
	   % First four values in the array are 3-face paths, the rest are four.
	   %------------------------------------------------------------------
	Dist = 1:12;
	
   % 3-Face Paths
	Dist(1) = sqrt((HR1-S1)^2+(HR2-S2)^2);
	Dist(2) = sqrt((HL1-S1)^2+(HL2-S2)^2);
	Dist(3) = sqrt((VU1-S1)^2+(VU2-S2)^2);
	Dist(4) = sqrt((VD1-S1)^2+(VD2-S2)^2);
	
   % 4-Face Paths
	Dist(5)  = sqrt((HR41A-S1)^2+(HR42A-S2)^2);
	Dist(6)  = sqrt((HR41B-S1)^2+(HR42B-S2)^2);
	Dist(7)  = sqrt((HL41A-S1)^2+(HL42A-S2)^2);
	Dist(8)  = sqrt((HL41B-S1)^2+(HL42B-S2)^2);
	Dist(9)  = sqrt((VU41A-S1)^2+(VU42A-S2)^2);
	Dist(10) = sqrt((VU41B-S1)^2+(VU42B-S2)^2);
	Dist(11) = sqrt((VD41A-S1)^2+(VD42A-S2)^2);
	Dist(12) = sqrt((VD41B-S1)^2+(VD42B-S2)^2);
	
	Dist       = round(Dist, 10);
	minDistVal = min(Dist);
	minIndex   = find(minDistVal == Dist);
	end
	\end{alltt}
}
\medskip\noindent\textbf{Plotting the cube function}
{\small
\begin{alltt}
function Plotting(minIndex, T1, T2)
   %------------------------------------------------------------------
   % Plotting checks to see if the min index < 4. If it is, the target points
   % (T1,T2) are plotted! Refer to the MinDistance function for each shortest
   % path index value.
   % Strict 4FSP's plotted in dots.
   %------------------------------------------------------------------
	if length(minIndex)==1
	if minIndex > 4
	scatter(T1,T2,"b."); %
	hold on;
	%end
   %------------------------------------------------------------------
   % If the strict 3FSP should be plotted in red, uncomment the following 
   % section and remove the prior end.
   %------------------------------------------------------------------  
	else
	scatter(T1,T2,".");
	hold on;
	end
	else
   %------------------------------------------------------------------  
   % Proper lax target points: when a 3FSP and its corresponding 4FSP are
   % the same length--plotted in green *
   %------------------------------------------------------------------  
	if 
	(ismember(1,minIndex) && (ismember(5,minIndex)  || ismember(6, minIndex)))  ||
	(ismember(2,minIndex) && (ismember(7,minIndex)  || ismember(8, minIndex)))  ||
	(ismember(3,minIndex) && (ismember(9,minIndex)  || ismember(10, minIndex))) ||
	(ismember(4,minIndex) && (ismember(11,minIndex) || ismember(12, minIndex)))
	scatter(T1,T2,"g*");
	hold on;
   %------------------------------------------------------------------    
   % Improper lax target points: when a 3FSP and a different 4FSP are the
   % same length--plotted in magenta *
   %------------------------------------------------------------------  
	elseif (improperLaxPoints(minIndex)==1) 
	scatter(T1, T2,"m*");
	hold on;
   % If there are 2 or more 3FSP's they are plotted in black * 
	elseif (threeShortestPath(minIndex)==1) 
	scatter(T1, T2,"k*");
	hold on;
	end
	end
\end{alltt}
}

\medskip\noindent\textbf{To determine 3FSP's}
{\small
\begin{alltt}
	function [result] = threeShortestPath(minIndex)
    %------------------------------------------------------------------  
    % threeShortest runs all conditions for when two (or more) 3-face paths are
    % shortest
    %------------------------------------------------------------------  
	if ismember(1,minIndex) && (ismember(2,minIndex) || ismember(3,minIndex) ||
	                            ismember(4,minIndex))
	result = 1;
	elseif ismember(2, minIndex) && (ismember(3,minIndex) || ismember(4,minIndex))
	result = 1;
	elseif ismember(3, minIndex) && ismember(4,minIndex)
	result = 1;
	else
	result=0;
	end
	end
	\end{alltt}
}

\medskip\noindent\textbf{To determine improper lax points}
{\small
\begin{alltt}
function [result] = improperLaxPoints(minIndex)
    % improperLaxPoints runs all the conditions for improper lax points,
	if (ismember(1,minIndex) && (ismember(7,minIndex)   || ismember(8, minIndex)  ||
	    ismember(9, minIndex) || ismember(10, minIndex) || ismember(11, minIndex) ||
	    ismember(12, minIndex)))
	result=1;
	elseif ismember(2,minIndex) && (ismember(5,minIndex)   || ismember(6, minIndex)  ||
	       ismember(9, minIndex) || ismember(10, minIndex) || ismember(11, minIndex) ||
	       ismember(12, minIndex))
	result=1;
	elseif ismember(3,minIndex) && (ismember(5,minIndex)  || ismember(6, minIndex)  ||
	       ismember(7, minIndex) || ismember(8, minIndex) || ismember(11, minIndex) ||
	       ismember(12, minIndex))
	result=1;
	elseif ismember(4,minIndex) && (ismember(5,minIndex)  || ismember(6, minIndex) ||
	       ismember(7, minIndex) || ismember(8, minIndex) || ismember(9, minIndex) ||
	       ismember(10, minIndex))
	result=1;
	else
	result=0;
	end
	
\end{alltt}
}

\medskip\noindent\textbf{To find the probabilities of 4FSP's for each source point}
{\small
\begin{alltt}
    %------------------------------------------------------------------  
    % CubeProbDist runs like CubeCode, except instead of plotting all of the
    % points that do not have only 3FSP's connecting them to the
    % given source point, this program keeps count of the number and divides it
   	% by the total number of points for the estimated probability of a 4FSP at 
   	% a given source point. 
    %------------------------------------------------------------------  
totalPoints=(2/.0025 + 1)^2; 
    %------------------------------------------------------------------  
   	% Note that the .0025 is the interval and can be adjusted to whatever 
   	% fineness desired, as long as all of the .0025 values are adjusted.
   	% Below are the vectors where the values are stored.
   	%------------------------------------------------------------------  
X = zeros(totalPoints,1);
Y = zeros(totalPoints, 1);
Prob = zeros(totalPoints,1);
currentPoint = 0; 
   	%------------------------------------------------------------------  
    % currentPoint is used to index the values that are stored in the 
    % vectors above.
    %------------------------------------------------------------------  
for S2 = -1:.0025:1   % S2 is the y-coord of the source point.
for S1 = -1:.0025:1   % S1 is the x-coord of the source point.
currentPoint = currentPoint + 1;
Pcount=0;
   	%------------------------------------------------------------------   
    % Pcount counts the number of target points that have 4FSP's to the source point, 
    % here it is being reset to 0 with each new source point.
    %------------------------------------------------------------------  
S1 = round(S1,4);   % To deal with precision errors.
S2 = round(S2,4);
    % The i and j represent the target points
for i = -1:.0025:1  % i is the "rows"
for j= -1:.0025:1   % j is the "columns"
i=round(i,4);       % To deal with precision errors
j=round(j,4);
	[HR1, HR2, HL1, HL2, VU1, VU2, VD1, VD2, HR41A, HR42A, HR41B, HR42B, HL41A,
	 HL42A,HL41B, HL42B, VU41A, VU42A, VU41B, VU42B, VD41A, VD42A, VD41B,
	 VD42B] = TargetCoords(j,i);
[minDistVal, minIndex] = MinDistance(S1, S2, HR1, HR2, HL1, HL2, VU1, VU2, 
                                     VD1, VD2, HR41A, HR42A, HR41B, HR42B, 
                                     HL41A, HL42A, HL41B, HL42B, VU41A, VU42A, 
                                     VU41B, VU42B, VD41A, VD42A, VD41B, VD42B);
Pcount = ProbCount(minIndex, j, i, Pcount);
end
end
X(currentPoint,1) = S1;
Y(currentPoint,1) = S2;
Prob(currentPoint,1) = (Pcount/totalPoints)*100;	end
if S2 == .5
fprintf("At .5")
end	
end
    %------------------------------------------------------------------
   	% We've later accessed the X, Y, and Prob in the workspace to run with
   	% contourplots or/and to store in an Excel file.
   	%------------------------------------------------------------------
	
\end{alltt}
}

\medskip\noindent\textbf{Keeping count of the 4FSP's for a source point} 
{\small
\begin{alltt}
function [Pcount] = ProbCount(minIndex, T1, T2, count) 
   	% ProbCount keeps a count of the number of 4FSP's on the cube.
if minIndex > 4
Pcount = count + 1;
    %------------------------------------------------------------------
   	% Proper lax target points: when a 3FSP and its corresponding 4FSP are
   	% the same length.
   	%------------------------------------------------------------------
elseif (ismember(1,minIndex) && (ismember(5,minIndex)  || ismember(6, minIndex)))  ||
       (ismember(2,minIndex) && (ismember(7,minIndex)  || ismember(8, minIndex)))  ||
       (ismember(3,minIndex) && (ismember(9,minIndex)  || ismember(10, minIndex))) ||
       (ismember(4,minIndex) && (ismember(11,minIndex) || ismember(12, minIndex)))
Pcount = count + 1; 	
elseif (improperLaxPoints(minIndex)==1) 
    %------------------------------------------------------------------
    % Improper lax target points: when a 3FSP and a different 4FSP are 
    % the same length.
    %--------------------------------------------------------------
Pcount = count + 1;
else
Pcount = count;
end
\end{alltt}
}

\medskip\noindent\textbf{Plotting the contour plot from data}
{\small
\begin{alltt}
    %------------------------------------------------------------------	
   	% Rearranging the data to be plotable with MATLAB's contour and contour3
   	% functions
   	% X, Y, and prob are column vectors of data that have been imported or in
   	% the workspace after running the CubeProbDist code.
   	%------------------------------------------------------------------
a = unique(X) ; nx = length(a) ;
b = unique(Y) ; ny = length(b) ;
Z = reshape(prob,ny,nx) ;
   	% For a 3d contour plot, add a 3 after the contour in the next line
c = contour(a,b,Z,50);
grid on
   	% Formatting the contour plot
pbaspect([1 1 1])
colormap(gray)
clabel(c,'manual');
\end{alltt}
}

\end{document}